\documentclass[reqno,12pt]{amsart}

\usepackage[utf8]{inputenc}

\usepackage{enumerate}
\usepackage[margin=1in]{geometry}
\usepackage{ifpdf}
\usepackage{amsmath}
\usepackage{amsfonts}
\usepackage{amssymb}
\usepackage{amsthm}
\usepackage[ocgcolorlinks,hyperfootnotes=false,colorlinks=true,citecolor=blue,linkcolor=blue,urlcolor=blue]{hyperref}
\usepackage{setspace}
\usepackage{amsrefs}
\usepackage{nicefrac}
\usepackage{graphicx}
\usepackage{color}
\usepackage{mathtools}

\newcommand{\ignore}[1]{}

\renewcommand{\Im}{\operatorname{Im}}

\newcommand{\sabs}[1]{\lvert {#1} \rvert}
\newcommand{\norm}[1]{\left\lVert {#1} \right\rVert}
\newcommand{\snorm}[1]{\lVert {#1} \rVert}

\newcommand{\C}{{\mathbb{C}}}
\newcommand{\R}{{\mathbb{R}}}

\newcommand{\sB}{{\mathcal{B}}}

\newcommand{\rank}{\operatorname{rank}}

\newtheorem{thm}{Theorem}[section]

\newtheorem{prop}[thm]{Proposition}

\newtheorem{cor}[thm]{Corollary}

\newtheorem{lemma}[thm]{Lemma}

\theoremstyle{definition}

\newtheorem{example}[thm]{Example}

\theoremstyle{remark}

\author{Ji\v{r}\'{\i} Lebl}
\thanks{The first author was in part supported by NSF grant DMS-1362337.}
\address{Department of Mathematics, Oklahoma State University,
Stillwater, OK 74078, USA}
\email{lebl@math.okstate.edu}

\author{Alan Noell}
\address{Department of Mathematics, Oklahoma State University,
Stillwater, OK 74078, USA}
\email{noell@math.okstate.edu}

\author{Sivaguru Ravisankar}
\address{Tata Institute of Fundamental Research, Centre for Applicable Mathematics, Bengaluru 560065, India}
\email{sivaguru@tifrbng.res.in}

\date{September 4, 2018}

\ifpdf
\hypersetup{
  pdftitle={On the Levi-flat Plateau problem},
  pdfauthor={Jiri Lebl, Alan Noell, Sivaguru Ravisankar},
  pdfsubject={Several Complex Variables},
  pdfkeywords={Levi-flat Plateau problem, CR singular, holomorphic hull, CR function},
}
\fi

\title{On the Levi-flat Plateau problem}


\begin{document}


\begin{abstract}
We solve the Levi-flat Plateau problem in the following case.  Let $M \subset {\mathbb C}^{n+1}$, $n \geq 2$, be a connected compact real-analytic codimension-two submanifold with only nondegenerate CR singularities.  Suppose $M$ is a diffeomorphic image via a real-analytic CR map of a real-analytic hypersurface in ${\mathbb C}^n \times {\mathbb R}$ with only nondegenerate CR singularities.  Then there exists a unique compact real-analytic Levi-flat hypersurface, nonsingular except possibly for self-intersections, with boundary $M$.  We also study boundary regularity of CR automorphisms of domains in ${\mathbb C}^n \times {\mathbb R}$.
\end{abstract}

\maketitle





\section{Introduction} \label{section:intro}

Let $M \subset \C^{n+1}$ be a real codimension-two compact manifold.  The
\emph{Levi-flat Plateau problem} asks:

\medskip

\emph{Does there exist a Levi-flat hypersurface $H$ with boundary $M$?}

\medskip

A Levi-flat hypersurface is a real codimension-one submanifold
that is locally foliated by complex hypersurfaces.
If $\rho =
0$ is a defining equation for a smooth Levi-flat hypersurface $H$, then $\rho$
satisfies a complex Monge--Amp{\`e}re type differential equation,
that is, a rank condition
on the complex bordered Hessian:
\begin{equation}
\rank
\left[
\begin{matrix}
\rho & \rho_z \\
\rho_{\bar{z}} & \rho_{z\bar{z}}
\end{matrix}
\right]
= 2
\ \ \ \text{ for all points on $H$. }
\end{equation}
The complex Hessian restricted to the complex tangent space
is called the Levi-form, and the equation says that the Levi-form of $H$
vanishes identically.  The problem is a minimal surface problem along the
lines of the Harvey--Lawson~\cite{HarveyLawson} solution to the 
complex Plateau problem,
which asks for a complex submanifold with a given boundary.  Since a Levi-flat
hypersurface is a one-dimensional family of complex hypersurfaces, the
Levi-flat Plateau problem is essentially the Harvey--Lawson problem with a
single real parameter.

In $\C^2$, that is, when $n=1$, the question is quite classical
and has been studied by many authors in both the local and the global setting:
Bishop \cite{Bishop:diffman},
Moser and Webster \cite{MW:normal}, Bedford \cite{bedford:boundaries}, 
Huang and Krantz \cite{huangkrantz}, or
Bedford and Gaveau \cite{BG:envhol}, among others.

In $\C^{n+1}$ when $n \geq 2$,
existence of $H$ requires local necessary conditions on $M$ at every CR
point:
$M$ must be ``flat'' in at least one direction.
The existence had been first tackled by
Dolbeault--Tomassini--Zaitsev~\cites{DTZ,DTZ2},
and the flattening (that is, local existence of $H$ near
the CR singularity) has been studied more recently by
Huang--Yin~\cite{HuangYin:flatI} and Fang--Huang~\cite{FangHuang}.
In the flat case, the authors have studied extension of CR
functions~\cites{crext1,crext2,crext3}.
In this present paper, we study the regularity
of an extended mapping.

In \cite{crext2} the authors proved that if $M$ is an image by a CR map of the
boundary of a bounded 
domain $\Omega \subset \C^n \times \R$, $n \geq 2$, where
$\Omega$ has connected real-analytic boundary with nondegenerate ($A$-nondegenerate,
see below) CR singularities,
then there exists a possibly singular Levi-flat hypersurface $H$ with
boundary $M$.  The question of regularity of $H$ was left open, however.
In this paper we prove the regularity of this solution for real-analytic
$M$.  For $n=1$, \cite{crext1} provides existence in the elliptic case under
a hypothesis replacing the CR condition.
We prove regularity in this case if the image is flat.

Let us be more specific.
(See \S\ref{section:prelim} for detailed
definitions.)
First, all submanifolds are assumed to be embedded submanifolds 
unless otherwise noted.
A real-analytic submanifold $M \subset \C^{n+1}$
of codimension two has a CR singularity at $p \in M$
if $M$ is not a complex submanifold near $p$
but is tangent to a complex hyperplane at $p$.
At a CR singularity $p$,
the submanifold can be written in local
coordinates $(z,w) \in \C^n \times \C$ vanishing at $p$ as
$w = \rho(z,\bar{z})$, where $\rho \in O(\snorm{z}^2)$.
We write this equation as
\begin{equation}
w = A(z,\bar{z}) + B(z,z) + \overline{C(z,z)} + E(z,\bar{z})
= Q(z,\bar{z}) + E(z,\bar{z}) ,
\end{equation}
where $A$ is a sesquilinear form, $B$ and $C$ are bilinear,
$E \in O(\snorm{z}^3)$, and $Q=A+B+\overline{C}$.
There are two natural nondegeneracy conditions on the CR singularity.
We could require $Q$ to be nondegenerate, or we could require that $A$ be nondegenerate. The second condition
is more appropriate
from the point of view of CR geometry since $A$ carries the
Levi-form information.
We call these two nondegeneracies \emph{$Q$-nondegenerate}
and \emph{$A$-nondegenerate}.

Note that for a codimension-two submanifold, isolated singularities
are the generic condition.
The forms $A$ and $C$ carry the local biholomorphic invariants, and in fact a
holomorphic change of coordinates can make $B=C$.
The CR singularity is said to be \emph{elliptic} if, after
a local change of coordinates, $A$ is Hermitian positive
definite, and the level sets of $A+C+\overline{C}$ are ellipsoids
($C$ has small eigenvalues relative to $A$).
An elliptic CR singularity is both $A$- and $Q$-nondegenerate.

We say $M$ is holomorphically flat if there exist coordinates
$(z,w) \in \C^n \times \C$ in which $\rho$ is real-valued.
We mostly deal with holomorphically flat manifolds, and
in fact much of the time we will restrict our attention to
$\C^n \times \R$, where we use coordinates $(z,s)$.
In this setting, if $\Omega \subset \C^{n} \times \R$
is a domain with smooth boundary, a smooth function
$\varphi \colon \overline{\Omega} \to \C$ is CR if $\varphi|_{\Omega}$ is
CR.  In this case $\frac{\partial \varphi}{\partial \bar{z}_k} = 0$ for all
$k$ at all points of $\Omega$, and hence at all points of $\overline{\Omega}$.

We are now ready to state our first result.

\begin{thm}\label{thm:main}
Let
$\Omega \subset \C^n \times \R$, $n \geq 2$, be a bounded domain
with connected real-analytic boundary
such that $\partial \Omega$ 
has only $A$-nondegenerate
CR singularities.
Let $\Sigma \subset \partial \Omega$ be the set of CR singularities of $\partial \Omega$.
Let $f\colon \partial \Omega \to \C^{n+1}$ be a real-analytic
embedding that is CR at CR points of $\partial \Omega$
and takes CR points of $\partial \Omega$ to CR points of $f(\partial \Omega)$.

Then, there exists a real-analytic CR map
$F\colon\overline{\Omega} \to\C^{n+1}$ such that
$F|_{\partial \Omega} = f$ and
$F|_{\overline{\Omega} \setminus \Sigma}$ is an immersion.
In other words, $F(\overline{\Omega})$ is the solution of the Levi-flat
Plateau problem for $f(\partial \Omega)$.
\end{thm}

There are two issues with extending the theorem to $n=1$.  First, the
CR condition is vacuous, and the extension $F$ may not in fact exist at all.
Simply take $\Omega \subset \C \times \R$ to be the unit ball, and
let $f(z,s) =(\bar{z},s)$. Then no such $F$ exists.
Even if an extension exists, the geometry of the CR singular points does not
force the derivative to be of full rank as in the case $n \geq 2$. 
See Example~\ref{example:nonimmersionlocal}.

The assumption that CR points go to CR points is necessary.  If  a CR point goes to a CR singular point, Example~\ref{example:crsingulartwo} shows that the image hypersurface may
be singular.  Note that, when a  CR point goes to a CR singular point, the image point cannot be
$A$-nondegenerate. See  Proposition~\ref{prop:imagenotnondegenerate}.



Next we use a theorem of
Fang--Huang~\cite{FangHuang}: When $n \geq 2$,
a real-analytic CR singular submanifold of $\C^{n+1}$ with an 
$A$-nondegenerate CR singularity, except for
one exceptional case,
can be ``flattened'' locally
if it is flat at the CR points.
Using this flattening and
the previous extension result of the
authors~\cite{crext2} on the inverse,
if we further assume that $f(\partial \Omega)$ has only
$A$-nondegenerate CR singularities as above,
then $F$ becomes an immersion on all of $\overline{\Omega}$,
including the CR singularities. 
The extension theorem (Theorem~\ref{thm:fanghuanglnr})
combining results of \cite{crext2} and Fang--Huang~\cite{FangHuang} may
be of independent interest.

We say a CR singular manifold $M \subset \C^{3}$ is the \emph{exceptional case}
if $M$ can be put into the form
\begin{equation} \label{eq:exceptionalcase}
w = \sabs{z_1}^2 - \sabs{z_2}^2
+ \lambda (z_1^2 + \bar{z}_1^2)
+ \lambda (z_2^2 + \bar{z}_2^2) + O(\snorm{z}^3)
,
\qquad \text{where $\lambda \geq \frac{1}{2}$} .
\end{equation}
This case is exceptional because this $M$ has no elliptic direction, whereas
all the other quadratically flat (where $A$ is real-valued) have an
elliptic direction (see~\cite{crext2}).

\begin{cor}  \label{cor:main}
Let
$\Omega \subset \C^n \times \R$, $n \geq 2$,
be a bounded domain with connected
real-analytic boundary
such that $\partial \Omega$ has only $A$-nondegenerate CR singularities,
and let $f\colon \partial \Omega \to \C^{n+1}$ be a real-analytic
embedding that is CR at CR points of $\partial \Omega$.
Assume that $f(\partial \Omega)$ has only $A$-nondegenerate 
CR singularities. Further assume that either $n \geq 3$ or
 no CR singularity of $f(\partial \Omega)$ is the exceptional case
(in other words, every CR singularity has an elliptic direction).

Then, there exists a real-analytic CR map
$F\colon\overline{\Omega} \to\C^{n+1}$ such that
$F|_{\partial \Omega} = f$ and  $F$ is an immersion
on ${\overline{\Omega}}$.
\end{cor}

As we mentioned above,
Proposition~\ref{prop:imagenotnondegenerate}
implies that
$f$ cannot take a CR point to an $A$-nondegenerate
CR singularity when $n\geq 2$,
so we do not need to assume that $f$ takes CR points to CR points.

Example~\ref{example:imagedegenerate} shows that the condition of
$A$-nondegeneracy of the image is necessary.  If the image $f(\partial \Omega)$
has an $A$-degenerate CR singularity, it is possible
for the derivative of $F$ at a CR singular point to be
rank-deficient.  Not only that, Example~\ref{example:cusp} shows that
the image $F(\overline{\Omega})$ need not be a real-analytic submanifold.

If, instead of applying a local flattening result 
such as the theorem of
Fang--Huang, we assume that the image of the boundary lies in
$\C^n \times \R$ (a global flattening),
then we strengthen the conclusions of the preceding results.
Furthermore,
we obtain an interesting corollary on the boundary behavior
of automorphisms of domains.

\begin{thm}\label{thm:flat}
Let
$\Omega \subset \C^n \times \R$, $n \geq 2$,
be a bounded domain with connected
real-analytic boundary
such that $\partial \Omega$ has only $A$-nondegenerate
CR singularities, and let
$f\colon \partial \Omega \to \C^n \times \R$ be a real-analytic
embedding that is CR at CR points of $\partial \Omega$.

Then, there exists a real-analytic CR map
$F\colon\overline{\Omega} \to\C^n \times \R$ such that
$F|_{\partial \Omega} = f$.

Furthermore:
\begin{enumerate}[(a)]
\item If $f$ takes CR points of $\partial \Omega$
to CR points of $f(\partial \Omega)$, then $F$ is
one-to-one on $\overline{\Omega}$, and
$F|_{\overline{\Omega} \setminus \Sigma}$ is an embedding. Here
 $\Sigma \subset \partial \Omega$ is the set of CR singularities of $\partial \Omega$.
\item If $f(\partial \Omega)$ has only $A$-nondegenerate
CR singularities,
then $F$ is one-to-one on $\overline{\Omega}$ and is an embedding on all of $\overline{\Omega}$.
\end{enumerate}
\end{thm}

Theorem \ref{thm:flat} has the following corollary about automorphisms of bounded
domains in $\C^n \times \R$. To state the result, we need to extend the notion of 
CR diffeomorphism to domains in
$\C^n \times \R$ with boundary.  For real-analytic boundaries and maps
we make the following definition.
Suppose $\Omega_1 , \Omega_2 \subset \C^n \times \R$
are domains with real-analytic boundary.
We say $F \colon \overline{\Omega}_1 \to \overline{\Omega}_2$
is a real-analytic CR diffeomorphism
if $F$ is onto and
if there exist neighborhoods $U_1,U_2$ in $\C^{n} \times \C$
such that $\overline{\Omega}_j \subset U_j$
and
a biholomorphic map $G \colon U_1 \to U_2$ such that
$G|_{\overline{\Omega}_1} = F$.
One can make an equivalent (but more complicated) intrinsic definition
that would extend to smooth manifolds if needed (see Proposition \ref{prop:equivcr}).

\begin{cor}\label{cor:auto}
Let
$\Omega \subset \C^n \times \R$, $n \geq 2$,
be a bounded domain with connected
real-analytic boundary
such that $\partial \Omega$ has only $A$-nondegenerate
CR singularities.
Let $f\colon \partial \Omega \to \partial \Omega$ be a real-analytic
automorphism, in the sense that $f$ is a diffeomorphism and is CR
at CR points (considered as a map into $\C^n \times \R$).

Then, there exists a real-analytic CR diffeomorphism
$F\colon\overline{\Omega} \to \overline{\Omega}$
such that
$F|_{\partial \Omega} = f$.
\end{cor}

The regularity part of Theorem \ref{thm:flat} still holds
in the case $n=1$ as long as
we replace the condition of being CR with assuming an
extension along each leaf, that is, the set where $s$
is fixed in the $(z,s)$ coordinates.
We also need to assume the singularities are elliptic.
In this setting, this assumption means that locally
in the $(z,s)$ coordinates the boundary can be expressed as
\begin{equation}
s = \sabs{z}^2 + \lambda z^2 + \lambda \bar{z}^2 + O(\sabs{z}^3)
\end{equation}
for $0 \leq \lambda < \frac{1}{2}$.
We remark that in this case $\partial \Omega$ is automatically connected.

\begin{thm} \label{thm:flatn1}
Let $\Omega \subset \C \times \R$ be a bounded domain
with real-analytic boundary
such that $\partial \Omega$ has only elliptic CR singularities, and
let $\Sigma \subset \partial \Omega$ be the set of CR singularities
of $\partial \Omega$.
Let $f\colon \partial \Omega \to \C \times \R$ be a real-analytic
embedding that
takes CR points of $\partial \Omega$ to CR points of
$f(\partial \Omega)$.
Suppose that for every $c \in \R$
such that $\Omega \cap \{ s = c \}$ is nonempty
there exists a continuous map on $\overline{\Omega} \cap \{ s = c \}$,
holomorphic on $\Omega \cap \{ s = c \}$,
extending $f|_{\partial \Omega \cap \{ s = c \}}$.

Then, there exists a one-to-one real-analytic CR map
$F\colon\overline{\Omega} \to\C \times \R$ such that
$F|_{\partial \Omega} = f$ and
$F|_{\overline{\Omega} \setminus \Sigma}$ is an embedding.

If furthermore $f(\partial \Omega)$ has only elliptic CR singularities,
then $F$ is an embedding on $\overline{\Omega}$.
\end{thm}

We also have the corresponding corollary for automorphisms.

\begin{cor} \label{cor:flatn1}
Let 
$\Omega \subset \C \times \R$ be a bounded domain with
real-analytic boundary
such that $\partial \Omega$ has only elliptic 
CR singularities.
Let $f\colon \partial \Omega \to \partial \Omega$ be a real-analytic
diffeomorphism.
Suppose that for every $c \in \R$ such that
$\Omega \cap \{ s = c \}$ is nonempty
there exists a continuous map on $\overline{\Omega} \cap \{ s = c \}$,
holomorphic on $\Omega \cap \{ s = c \}$,
extending $f|_{\partial \Omega \cap \{ s = c \}}$.

Then, there exists a real-analytic CR diffeomorphism
$F\colon\overline{\Omega} \to \overline{\Omega}$
such that
$F|_{\partial \Omega} = f$.
\end{cor}

Let us outline the organization of the paper.  First, in
\S\ref{section:prelim} we provide detailed definitions and
some basic results.
In \S\ref{section:examples} we provide several examples
to prove the necessity of the hypotheses in our theorems and to illustrate
the issues encountered.  In \S\ref{section:proofofmain} and
\S\ref{section:proofofcor} we prove
Theorem~\ref{thm:main} and Corollary~\ref{cor:main}, respectively.  Then
in \S\ref{section:flat} and \S\ref{section:flatn1}
we prove the results in the case of a globally flat target.

\section{Preliminaries}\label{section:prelim}

Let us precisely define the terms used.
First let us start with the CR structure.
As noted earlier, we will generally deal with smooth
real submanifolds in $\C^n \times \R$ using coordinates $(z,s)$,
or in $\C^n \times \C = \C^{n+1}$ with coordinates $(z,w)$.
We consider $\C^n \times \R$ naturally embedded in $\C^{n+1}$.
Let $J$ denote the complex structure on the tangent bundle $T \C^{n+1}$.
Let $M$ be a submanifold and $T_p M$ its tangent space at $p \in M$.
Denote by
\begin{equation}
    T_p^c M = T_pM \cap J(T_pM)
\end{equation}
the \emph{complex tangent space} (a real vector space).
We also decompose the complexified $T_p^c M$ into holomorphic
and antiholomorphic vectors:
\begin{equation}
    \C \otimes T_p^c M = T_p^{1,0} M \oplus T_p^{0,1} M .
\end{equation}
A submanifold is said to be CR at a point $p$ if $T_q^c M$ (or $T_q^{0,1} M$) has
constant dimension as $q$ varies near $p$.  Otherwise, $M$ is said to
be CR singular at $p$.  For a $2n$-dimensional submanifold of
$\C^{n+1}$ there are only two possibilities: $T_p^{0,1} M$
is of complex dimension $n$ or $n-1$, and this dimension is called
the CR dimension.  Generically a real submanifold has CR dimension $n-1$,
and generically CR singularities are isolated.  When the CR dimension is locally
$n$,  $M$ is a complex manifold (the Newlander--Nirenberg theorem).

For $2n$-dimensional submanifolds of the Euclidean Levi-flat
$\C^n \times \R$,
the CR structure is relatively easy to describe.
The subsets of the form $\{ (z,s) : s = s_0 \}$ are called
leaves of this Levi-flat.  The dimension of $T_p^{0,1} M$ is
$n$ if and only if $M$ is tangent to a leaf at $p$.

Let $M \subset \C^{n+1}$ be a $2n$-dimensional submanifold that is not a complex manifold, and assume $0 \in M$.  The manifold
$M$ has a CR singularity at
the origin if it is tangent to a complex hypersurface
at the origin.  After a rotation we may assume
the coordinates are
$(z,w) \in \C^{n} \times \C$
and
$M$ is tangent to the complex hypersurface $w=0$
at the origin.
Then near the origin $M$ is defined by
\begin{equation} \label{eq:CRsingulargeneral}
\begin{split}
  w & = \rho(z,\bar{z}) =
  A(z,\bar{z}) + B(z,z) + \overline{C(z,z)} + O(\snorm{z}^3)
  \\
  & = Q(z,\bar{z}) + O(\snorm{z}^3),
\end{split}
\end{equation}
where $A$ is a sesquilinear form and $B$ and $C$
are bilinear forms.
It is not difficult to show that any biholomorphic change of coordinates
preserving the form \eqref{eq:CRsingulargeneral} transforms the
matrix defining $A$ by $*$-congruence and multiplication by constants.
On the other hand, the form $B$ can be changed arbitrarily via
a biholomorphic change of coordinates.
Therefore, if we are working up to a change of coordinates,
it may be convenient to set $C = B$.

We are mostly interested in manifolds that are
\emph{holomorphically flat}, that
is, ones that can be put into a form where $\rho$ is real-valued via a
biholomorphic map.
A necessary (but not sufficient) condition for $M$
being holomorphically flat
is that after a change of variables $A$ is a Hermitian
form, that is, $A(z,\bar{z})$ is real-valued.
As we mentioned above, we can
make the rest of the quadratic terms real-valued by making $C=B$.

We will say that the CR singularity is \emph{$A$-nondegenerate} if $A$ is
nondegenerate, that is, the underlying matrix is nonsingular.
Similarly, we will say that the CR singularity is
\emph{$Q$-nondegenerate} if $Q$ is
nondegenerate as a real-bilinear form.
From the point of view of several complex variables,
it is most convenient to assume that $A$ is nondegenerate,
since it carries the Levi-form information of $M$ for points near
the origin.

A CR singularity is \emph{elliptic}
if after changing coordinates to make $B=C$,
we can make $A$ Hermitian positive definite by a rotation in the $w$ coordinate,
and
for any $s_0$ the sets
\begin{equation}
\left\{ z \in \C^n : s_0 =  A(z,\bar{z}) + B(z,z) +
\overline{B(z,z)} \right\}
\end{equation}
are compact (empty, isolated points, or ellipsoids).
By simultaneously diagonalizing $A$ and $B$
(possible if $A$ is positive definite),
at an elliptic CR singularity the manifold
can be put into the form
\begin{equation} \label{eq:typeofsing2}
w = \sum_{j=1}^{n} \bigl( \sabs{z_j}^2 + \lambda_j (z_j^2 + \bar{z}_j^2)
\bigr) +
O(\snorm{z}^3),
\qquad \text{where $0 \leq \lambda_j < \frac{1}{2}$}.
\end{equation}

We also need some preliminary results.
Some of the following propositions do not require
real-analyticity, so we state them
for smooth ($C^\infty$) manifolds and maps.

\begin{prop} \label{prop:sizeofCRsing}
Suppose $M \subset \C^n \times \R$ is a $2n$-dimensional real-analytic
CR singular submanifold
with a CR singularity at $p \in M$.  The set $\Sigma \subset M$
of CR singularities is a real-analytic subvariety, and
\begin{enumerate}[(a)]
\item
If $M$ is $Q$-nondegenerate at $p$, then $p$ is an isolated
CR singularity.
\item
If $M$ is $A$-nondegenerate at $p$, then the germ $(\Sigma,p)$
has real dimension at most $n$.
\end{enumerate}
\end{prop}

\begin{proof}
We work locally near $p$.
Suppose $p=0$ and as before write the manifold as
\begin{equation}
s = A(z,\bar{z}) + B(z,z) + \overline{B(z,z)} +E(z,\bar{z})
= Q(z,\bar{z}) + E(z,\bar{z}) = \rho(z,\bar{z}) .
\end{equation}
Parametrize $M$ by $z$.
The set of CR singularities is where the manifold is tangent to
a leaf $\{ s = s_0 \}$.
That is, $\Sigma$ is the set where $\nabla \rho = 0$.
In particular, it is a real-analytic set.
In \cite{crext2}*{Lemma 2.1}, we found that $\nabla Q = 0$ defines
a real subspace of dimension at most $n$ if $A$
is nondegenerate.
(Lemma 2.1 in \cite{crext2} states $n \geq 2$, but
the proof holds for $n=1$.)
If $Q$ is nondegenerate as a real bilinear form,
then $\nabla Q = 0$ only at the origin.
As $\rho = Q+E$ and $E$ vanishes to higher order, the result
follows, that is, the germ at 0 of the set defined
by $\nabla \rho = 0$ cannot be of larger dimension than the
set defined by $\nabla Q = 0$.
A quick way to see this is to notice that $\nabla Q$
is linear, and the zero set of $\nabla \rho$ is a
subset of the zero set where we consider only the components
corresponding to nonzero eigenvalues of $\nabla Q$.
\end{proof}

\begin{prop} \label{prop:CRdiffeo}
If $M \subset \C^m$ and $\widetilde{M} \subset \C^m$ are
smooth CR submanifolds of
the same CR dimension
and $f \colon M \to \widetilde{M}$ is
a CR map that is a diffeomorphism onto $\widetilde{M}$, then
$f$ is a CR diffeomorphism. In other words, the inverse is a CR map.
\end{prop}

\begin{proof}
Because the map $f$ is CR,
it takes
$T_q^{(1,0)} M$ to $T_{f(q)}^{(1,0)} \widetilde{M}$ for all $q \in M$.
As $f$ is a diffeomorphism,
the total derivative is invertible. Thus $f^{-1}$ takes
$T_{f(q)}^{(1,0)} \widetilde{M}$ onto $T_q^{(1,0)} M$.
\end{proof}

Let us show that a real-analytic CR function on the closure of 
a domain with real-analytic boundary
in $\C^n \times \R$ complexifies in the correct manner.

\begin{prop} \label{prop:flatCRcomplexify}
Let $\Omega \subset \C^n \times \R$, $n \geq 1$,
be a domain with real-analytic boundary,
and let 
$F \colon \overline{\Omega} \to \C$ be real-analytic such that
$F|_{\Omega}$ is CR.
Then there exist a domain $U \subset \C^n \times \C$ with $\overline{\Omega} \subset U$
and a holomorphic function $G \colon U \to \C$ such that $G|_{\overline{\Omega}} = F$.
\end{prop}

As we mentioned in the introduction we will normally just say that $F$ is CR,
which is equivalent
to $F|_{\Omega}$ being CR.

\begin{proof}
As $F$ is real-analytic up to the boundary, there exist a domain $W \subset \C^n \times \R$
with $\overline{\Omega} \subset W$ and a real-analytic
function $\widetilde{F} \colon W \to \C$
such that $\widetilde{F}|_{\overline{\Omega}} = F$.  Because $\widetilde{F}$ is real-analytic and CR
on an open subset of $W$,  it is CR on all of $W$.  Therefore, the
real-analytic $\widetilde{F}$ complexifies to a function $G$ on a domain $U$ as in the statement.
\end{proof}

Let us show that an intrinsic definition of
a real-analytic CR diffeomorphism of domains with
boundaries in $\C^n \times \R$ is equivalent to the one
given in the introduction.

\begin{prop} \label{prop:equivcr}
Suppose $\Omega_1 , \Omega_2 \subset \C^n \times \R$
are domains with real-analytic boundary.
Let $F \colon \overline{\Omega}_1 \to \overline{\Omega}_2$
be a bijective real-analytic map
(up to the boundary) such that $F^{-1}$ is real-analytic
and 
$F|_{\Omega_1}$ is CR.

Then, there exist neighborhoods $U_1,U_2$ in $\C^{n} \times \C$
such that $\overline{\Omega}_j \subset U_j$, and
a biholomorphic map $G \colon U_1 \to U_2$ such that
$G|_{\overline{\Omega}_1} = F$.
\end{prop}

\begin{proof}
First, Proposition~\ref{prop:CRdiffeo} says
that $F^{-1}|_{\Omega_2}$ is CR.  We then
apply Proposition~\ref{prop:flatCRcomplexify} to both $F$ and $F^{-1}$.
That is, we complexify $F$ to a neighborhood $\widetilde{U}_1$ of
$\overline{\Omega}_1$ in $\C^{n+1}$ and get
$G \colon \widetilde{U}_1 \to \C^{n+1}$,
where $G|_{\overline{\Omega}_1} = F$.
We complexify $F^{-1}$ to a neighborhood $\widetilde{U}_2$ of
$\overline{\Omega}_2$ in $\C^{n+1}$ and get
$H \colon \widetilde{U}_2 \to \C^{n+1}$,
where $H|_{\overline{\Omega}_2} = F^{-1}$.

Because $\overline{\Omega}_1 \subset G^{-1}(\widetilde{U}_2)$,
we may pick $U_1$ to be a connected neighborhood of
$\overline{\Omega}_1$ in $\C^{n+1}$ such that
$U_1 \subset G^{-1}(\widetilde{U}_2)$ and
such that $U_1 \cap \C^{n} \times \R$
is connected.  From now on we consider $G \colon U_1 \to \C^{n+1}$
to be the restriction to $U_1$.
Let $W_1 = U_1 \cap \C^n \times \R$
and $W_2 = \widetilde{U}_2 \cap \C^n \times \R$.

As $F^{-1} \circ F = Id$, then
$(H|_{W_2} \circ G|_{W_1})|_{\overline{\Omega}_1} = Id$. Since
$W_1$ is connected and $H|_{W_2}$ and $G|_{W_1}$ are real-analytic,
we find $H|_{W_2} \circ G|_{W_1} = (H \circ G)|_{W_1} = Id$.
Because $U_1$ is connected, $W_1$ is a real hypersurface in $U_1$, and
$H$ and $G$ are analytic, we find
$H \circ G = Id$ on $U_1$. The proposition follows by
taking $U_2 = G(U_1)$.
\end{proof}

\begin{prop} \label{prop:imagenotnondegenerate}
Suppose $M \subset \C^n \times \R$, $n \geq 2$,
is a smooth CR submanifold of real dimension $2n$,
and $\widetilde{M} \subset \C^{n+1}$ is a smooth
submanifold of real dimension $2n$ with an $A$-nondegenerate 
CR singularity at
$\widetilde{p} \in \widetilde{M}$.

Then there does not exist a germ at a point $p\in M$ of
a smooth CR map $f \colon M \to \C^{n+1}$ such that
$f(p)=\widetilde{p}$, $f(M) = \widetilde{M}$ (as germs),
and $f$ is a diffeomorphism onto its image.
\end{prop}

The proposition does not hold for $n = 1$.
See Example~\ref{example:crsingularlocal}.

\begin{proof}
Recall that a smooth CR submanifold of $\C^n \times \R$
of real dimension $2n$ either has CR dimension $n-1$
or is a complex submanifold.
If $M$ is a complex manifold, $f$ is holomorphic with a nonsingular
derivative, and $f(M)$ is an immersed complex manifold. Hence, $f(M)$
is not
a CR singular submanifold.
Therefore, assume that the CR dimension of $M$ is $n-1$.

Suppose there exists such a map (as a germ at $p$).
We may suppose that $p=0$ and $\widetilde{p}=0$, and
we write  $\widetilde{M}$ as  
\begin{equation}
\eta = \rho(\xi,\bar{\xi}) = A(\xi,\bar{\xi}) + B(\xi,\xi) + \overline{C(\xi,\xi)} + O(\snorm{\xi}^3),
\end{equation}
where $A$ is a nondegenerate sesquilinear form and $B$ and $C$
are bilinear forms.

The CR vector fields at CR points of $\widetilde{M}$ are spanned by
\begin{equation}
L_{j,k} = \rho_{\bar{\xi}_k}\partial_{\bar{\xi}_j}-\rho_{\bar{\xi}_j}\partial_{\bar{\xi}_k}+
\big(\rho_{\bar{\xi}_k}\bar{\rho}_{\bar{\xi}_j}-\rho_{\bar{\xi}_j}\bar{\rho}_{\bar{\xi}_k}\big)\partial_{\bar{\eta}}
\end{equation}
for $1\le j,k\le n$.

Let $(z,s) \in \C^n \times \R$ be the coordinates on the source.
As $M$ is CR, and not complex,
$g(z,s) = s$ as a function on $M$ is a real-valued
CR function with nonzero derivative.
Let  $\varphi=g\circ f^{-1}$.
Then, $\varphi$ is a real-valued CR function on $\widetilde{M}_{CR}$ such
that the derivative of $\varphi$ does not
vanish at the origin.

We write $\varphi$ as a function of $\xi$ and $\bar{\xi}$, since $\eta$
is a function of $\xi$ and $\bar{\xi}$.
Then, for CR  vector fields $L$ on $\widetilde{M}_{CR}$,
we have $L\varphi= 0$ and
$\bar{L}\varphi=\overline{L\bar{\varphi}}=\overline{L\varphi}=0$
on $\widetilde{M}_{CR}$.
Write $L=L_{j,k}$ for some $1\le j,k\le n$.  Then
$L\varphi$ and $\bar{L}\varphi$ are functions of
$\xi$ and $\bar{\xi}$ only, and
looking at the linear terms in $\xi$ and $\bar{\xi}$, we find
\begin{equation}
\begin{aligned}
& 0 = 
L \varphi = 
(A_{\bar{\xi}_k} + \overline{C_{\xi_k}})
\frac{\partial \varphi}{\partial {\bar{\xi}_j}}(0,0)
-
(A_{\bar{\xi}_j} + \overline{C_{\xi_j}})
\frac{\partial \varphi}{\partial {\bar{\xi}_k}}(0,0)
+ O(\snorm{\xi}^2),\\
& 0 = 
\bar{L} \varphi = 
(A_{\xi_k} + B_{\xi_k})
\frac{\partial \varphi}{\partial {\xi_j}}(0,0)
-
(A_{\xi_j} + B_{\xi_j})
\frac{\partial \varphi}{\partial {\xi_k}}(0,0)
+ O(\snorm{\xi}^2).
\end{aligned}
\end{equation}
The subscripts on the forms denote derivatives, treating the forms as
functions.
As everything is in terms of $\xi$ and $\bar{\xi}$, and these
parametrize $\widetilde{M}$, then the
above identities are true for all $(\xi,\bar{\xi})$.
The function $A_{\bar{\xi}_k}$ is the $k$th row of $A$
but with the $\ell$th
component multiplied by $\xi_\ell$.
On the other hand,
$\overline{C_{\xi_k}}$ depends only on $\bar{\xi}$.
In the second equation, $A_{\xi_k}$ involves the $k$th column of $A$
and depends on $\bar{\xi}$, whereas $B_{\xi_k}$ depends only on
$\xi$.
We complexify the Taylor polynomial, and assume
$\xi$ and $\bar{\xi}$ are independent variables.
As $A$ is nondegenerate, its rows are independent, and hence
using the first equation we get
$\frac{\partial \varphi}{\partial {\bar{\xi}_j}}(0,0) = 0$ and
$\frac{\partial \varphi}{\partial {\bar{\xi}_k}}(0,0) = 0$.
Similarly, as the columns of $A$ are independent, using the
second equation we find
$\frac{\partial \varphi}{\partial {\xi_j}}(0,0) = 0$ and
$\frac{\partial \varphi}{\partial {\xi_k}}(0,0) = 0$.
Consequently, the derivative of $\varphi$ must vanish at $(0,0)$, which is a contradiction.
\end{proof}

Next we prove the simple result that when a holomorphic map restricted to
a real submanifold is a diffeomorphism onto its image, the map can only raise
the CR dimension at a point.
The proposition holds also for smooth CR maps on CR manifolds;
however, the whole point is to
apply it in a situation where $M$ is CR singular at $p$.

\begin{prop}
Suppose $M \subset U \subset \C^m$ and $N \subset \C^k$,
where $U$ is open and
$M$ and $N$ are embedded real smooth submanifolds.
Suppose $F \colon U \to \C^k$ is holomorphic such that
$F(M) \subset N$ and $F|_{M}$ is a diffeomorphism onto its image.
Then for any $p \in M$
\begin{equation}
\dim T_p^{1,0} M \leq \dim T_{F(p)}^{1,0} N .
\end{equation}
\end{prop}

\begin{proof}
The map $F$ is holomorphic, so it takes
$T_p^{1,0} \C^m$ to $T_{F(p)}^{1,0} \C^k$, and hence
$T_p^{1,0} M$ to $T_{F(p)}^{1,0} N$. In other words,
the pushforward satisfies $F_* T_p^{1,0} M \subset T_{F(p)}^{1,0} N$.
The map $F$ is a diffeomorphism restricted to
$M$, so the (complexified) derivative is nonsingular
on $T_p^{1,0} M$, and thus
\begin{equation}
\dim T_p^{1,0} M = 
\dim F_* T_p^{1,0} M \leq
\dim T_{F(p)}^{1,0} N .
\end{equation}
\end{proof}

A corollary is the following counterpart of
Proposition~\ref{prop:imagenotnondegenerate}.

\begin{cor} \label{cor:imageofnondegeneratenotCR}
Suppose $M \subset \C^n \times \R$, $n \geq 2$, is a real-analytic
submanifold of dimension $2n$ with an $A$-nondegenerate
CR singularity at $p \in M$, $\widetilde{M} \subset \C^{n+1}$ is a smooth
CR submanifold of dimension $2n$ that has CR dimension $n-1$,
and $\widetilde{p} \in \widetilde{M}$.

Then there does not exist a germ at $p$ of a
real-analytic map $f \colon M \to \C^{n+1}$ such that
$f(p)=\widetilde{p}$, $f$ is CR at CR points of $M$,
$f(M) = \widetilde{M}$ (as germs),
and $f$ is a diffeomorphism onto its image.
\end{cor}

\begin{proof}
Suppose for a contradiction that $f$ exists.
Using the results of \cite{crext2} we have that
$f$ extends to a holomorphic map on a neighborhood of $p$.
The CR dimension of $M$ at $p$ is $n$, and
the CR dimension of $\widetilde{M}$ at $\widetilde{p}$ is $n-1$,
contradicting the previous proposition.
\end{proof}

The following result shows how we use the assumption
that CR points go to CR points.

\begin{prop} \label{prop:MimmersionatCR}
Suppose $M \subset \C^n \times \R$, $n \geq 1$, is a real-analytic CR
submanifold of dimension $2n$ that is not a complex manifold.
Suppose there exist
a neighborhood $U \subset \C^{n+1}$ of $p \in M$ and a holomorphic
map $G \colon U \to \C^{n+1}$ such that $G|_{M \cap U}$ is an immersion and
$G(p)$ is a CR point of $G(M \cap U)$.
Then $G$ is an immersion in a neighborhood of $p$.
\end{prop}

\begin{proof}
The derivative of a holomorphic map at $p$
is determined by a linear map on a maximally totally-real
subspace $V \subset T_p \C^{n+1}$.  Because
$M$ is of codimension two in $\C^{n+1}$ and not complex,
$M$ is generic, that is,
$T_p M + J T_p M = T_p \C^{n+1}$.
Therefore, there exists
such a maximally totally-real $V \subset T_p M$.
As $G|_{M \cap U}$ is an immersion at $p$, its derivative
is nonsingular on $V$.
Also, $G(M \cap U)$ is generic near $G(p)$ since $G(p)$ is a
CR point of $G(M \cap U)$.
Therefore, the derivative of $G$ is an invertible
map from $T^c_p M$ to $T^c_{G(p)} G(M\cap U)$ that preserves the
complex structure.
It follows that the pushforward of $V$ is a maximally totally-real
subspace of $T_{G(p)} \C^{n+1}$. Hence,
the derivative of $G$ at $p$ is nonsingular.
\end{proof}

\section{Examples} \label{section:examples}

Let us discuss some examples to see why the conditions in the theorems are
necessary.
In all of the following examples, as in the above results,
$M \subset \C^n \times \R$
will be the
source manifold in coordinates $(z,s)$,
either CR singular or not, $f \colon M \to \C^{n+1}$
will be a restriction of a holomorphic map $F$
(a map depending on $z$, $s$, but
not $\bar{z}$), and $f$ will be a diffeomorphism onto its image $f(M)$.
When we talk about a domain $\Omega \subset \C^{n} \times \R$ then 
$M=\partial \Omega$.

First, in the case $n=1$, in contrast to the higher-dimensional
case, the local geometry near an $A$-nondegenerate CR singular
point does not
force $F$ to be an immersion.

\begin{example} \label{example:nonimmersionlocal}
Let $M \subset \C \times \R$ be given by
\begin{equation}
s = \sabs{z}^2 ,
\end{equation}
and take $F(z,s) = (z,zs+s^2)$.  Then the origin is an elliptic CR singularity,
and the map $F$ restricted to $M$ is a diffeomorphism,
but $F$ is a finite map, not an immersion.
The map $F(z,s) = (z,zs)$ has the same properties except
it is not even finite.
\end{example}

An interesting question is whether a global example as above exists.
That is, do there
exist a bounded domain with real analytic boundary $\Omega \subset \C \times \R$,
where $\partial \Omega$ has only  elliptic CR singularities, and
a CR map $F \colon \overline{\Omega} \to \C^2$,
such that $F|_{\partial \Omega}$ is a diffeomorphism
onto its image, but such that $F$ is not an immersion?

A further complication is that
if $n=1$, then the image of a CR submanifold can be a
nondegenerate ($A$- and $Q$-nondegenerate in our sense and in the Bishop sense)
CR singular submanifold.

\begin{example} \label{example:crsingularlocal}
Take the CR (totally-real)
$M \subset \C \times \R$  given by
\begin{equation}
\Im z = 0 .
\end{equation}
The image $f(M)$ under $(z,s) \overset{f}{\mapsto} (z+is,s^2+z^2)$
is the manifold in the
variables $(\xi,\eta) \in \C^2$ defined by
\begin{equation}
\eta = \sabs{\xi}^2 ,
\end{equation}
which is an elliptic CR singular Bishop surface,
$A$- and $Q$-nondegenerate in our sense.

In fact, any function on a totally-real
CR submanifold is trivially CR, so
 any codimension-two submanifold of $\C^2$ is locally an image
of a totally-real submanifold via a CR embedding.
\end{example}

Proposition~\ref{prop:imagenotnondegenerate}
says there is no example as above for $n \geq 2$
and an $A$-nondegenerate image.
However, if we do not assume $A$-nondegeneracy,
it is possible for a CR point
to go to a CR singular point even when $n \geq 2$.

\begin{example} \label{example:crsingulartwo}
Take the CR manifold $M$ in $\C^2 \times \R$ defined by
\begin{equation}
s = z_1 + \bar{z}_1 + \sabs{z_2}^2.
\end{equation}
The CR vector field is
\begin{equation}
X =
z_2\frac{\partial}{\partial \bar{z}_1} - 
\frac{\partial}{\partial \bar{z}_2} .
\end{equation}
Let $(z,w)$ be the coordinates of $\C^2 \times \C$ and
think of $s$ as the real part of $w = s + it$ to make $M$ a codimension 2
submanifold of
$\C^3$.
The Levi-form of $M$ (in $\C^3$) is represented by
two $3 \times 3$ matrices, where one of them is identically zero.
The other matrix representing the Levi-form is
\begin{equation}
L = 
\begin{bmatrix}
0 & 0 & 0 \\
0 & 1 & 0 \\
0 & 0 & 0
\end{bmatrix}
\end{equation}
and $X^* L X = 1$ for the CR vector field above.
Therefore, the manifold $M$ is not Levi-flat (the Levi-form is nonzero),
and the Levi-flat hypersurface
$\{ t=0 \}$ that $M$ lies in is the unique such hypersurface at each point.
The manifold 
therefore has a real codimension-one foliation by the CR orbits.

The image $f(M)$ under $(z,s) \overset{f}{\mapsto} (z,s^2+is^3)$
is the manifold in the
variables $(\xi,\sigma+i\tau) \in \C^3$ defined by 
\begin{equation}
\sigma^3=\tau^2 \qquad \text{and} \qquad \tau = (\xi_1+\bar{\xi}_1+\sabs{\xi_2}^2) \sigma ,
\end{equation}
or equivalently
\begin{equation}
\sigma = (\xi_1+\bar{\xi}_1+\sabs{\xi_2}^2)^2
\qquad \text{and} \qquad
\tau = (\xi_1+\bar{\xi}_1+\sabs{\xi_2}^2)^3,
\end{equation}
which is a CR singular surface.
The map is a diffeomorphism from $M$ onto $f(M)$:
if we parametrize
$M$ by $z$ and $f(M)$ by $\xi$, $f$ is the identity.
The CR dimension of $f(M)$ outside of the CR singularity is 1, the same
as $M$.
As the map is CR and a diffeomorphism,
and the CR dimensions are the same,
by Proposition \ref{prop:CRdiffeo} the map is
a CR diffeomorphism when restricted
to the preimage of CR points of $f(M)$.
Since $M$ is not Levi-flat, $f(M)$ is not Levi-flat at CR points,
and the Levi-flat hypersurface $\{ \sigma^3=\tau^2 \}$ is
the unique Levi-flat hypersurface that contains $f(M)$.

Proposition~\ref{prop:imagenotnondegenerate} says that the image $f(M)$ is
an $A$-degenerate CR singular manifold.

Therefore:
\begin{enumerate}[1)]
\item The image of a CR point can be CR singular,
      although not $A$-nondegenerate.
\item The Levi-flat hypersurface containing $f(M)$ at such a point
      can be singular.
\end{enumerate}
\end{example}

Let us therefore move to a CR singular manifold and look at its image.
Even if the source is an $A$- or $Q$-nondegenerate CR singular manifold,
the image 
may in fact be degenerate in both senses.

\begin{example} \label{example:imagedegenerate}
Define $M$ in $\C^n \times \R$, $n \geq 1$, by
$s = \norm{z}^2$.  Take $f$ defined by
 $(z,s) \mapsto (z,s^2)$.
The image $f(M)$ in $\C^n \times \R$ with coordinates $(\xi,\sigma)$ is
\begin{equation}
\sigma = \norm{\xi}^4 ,
\end{equation}
which is degenerate in every sense.
\end{example}

The extended mapping $F$ of the above example is not an immersion at the CR
singularity; however, the image $f(M)$ still lies in the nonsingular
Levi-flat hypersurface given by $\Im \sigma = 0$.  It is possible
to construct an example where $f(M)$ does not
lie in any nonsingular Levi-flat.
In the next example we define
a CR singular manifold $M$ with an elliptic CR singularity
at the origin
and a real-analytic embedding $f$ that is CR at CR points such that $f$ 
maps the origin to the origin and 
the image $f(M)$ does not lie in a nonsingular Levi-flat
real-analytic hypersurface in any neighborhood of the origin.
The image $f(M)$ has an $A$- and $Q$-degenerate CR singularity.

\begin{example} \label{example:cusp}
Take the CR singular elliptic manifold $M$ in $\C^n \times \R$, $n \geq 1$, defined by
\begin{equation}
s = \norm{z}^2 .
\end{equation}
The image $f(M)$ under $(z,s) \overset{f}{\mapsto} (z,s^2+is^3)$
is the manifold in the
variables $(\xi,\sigma+i\tau)\in\C^{n+1}$ defined by 
\begin{equation}
\sigma^3=\tau^2 \qquad \text{and} \qquad \tau = (\norm{\xi}^2) \sigma ,
\end{equation}
or in other words
\begin{equation}
\sigma = \norm{\xi}^4
\qquad \text{and} \qquad
\tau = \norm{\xi}^6 ,
\end{equation}
which is $A$- and $Q$-degenerate.  The singular Levi-flat hypersurface
$\{ \sigma^3=\tau^2 \}$ is the
unique Levi-flat that contains $f(M)$, and it is singular at the origin.
The hypersurface on one side of $f(M)$ is
a $C^1$ manifold with boundary at the origin, namely
$\tau = \sigma^{3/2}$ (for $\sigma \geq 0$), but it is not a $C^2$ submanifold.
See Figure~\ref{fig:cusp}.

\begin{figure}[ht!]
\includegraphics[scale=1]{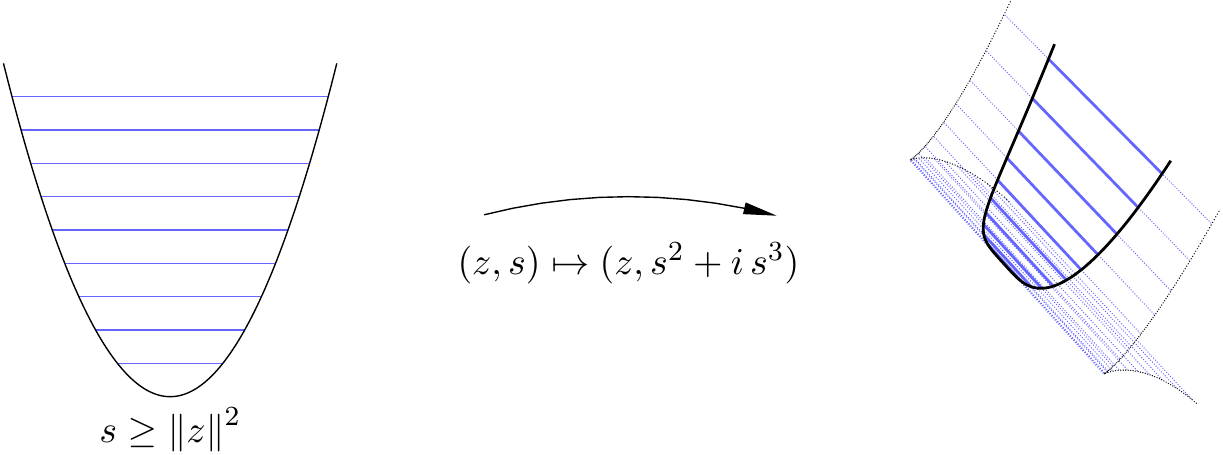}
\caption{Diagram of Example~\ref{example:cusp}.\label{fig:cusp}}
\end{figure}

The submanifold $f(M)$ is also an example of the necessity of 
nondegeneracy in a flattening result such as the theorem of
Fang--Huang.
\end{example}

Next let us focus on the injectivity of the extension $F$.
A map from
$\overline{\Omega} \subset \C^n \times \R$, $n \geq 1$, 
may not be one-to-one even if the map is
a diffeomorphism on the boundary and
the image lies in $\C^n \times \R$, if we allow the image to have further
CR singularities.

\begin{example}\label{eg:BallSingSq}
Let 
$\Omega \subset \C^n \times \R$, $n \geq 1$, be the slightly shifted unit
ball
\begin{equation}
\snorm{z}^2 + (s+\epsilon)^2 < 1 ,
\end{equation}
for some small $\epsilon > 0$.
Take $f \colon \partial \Omega \to \C^{n} \times \R$ defined by
$(z,s) \mapsto (z,s^2)$.
The image in $\C^{n} \times \R$ with coordinates $(\xi,\sigma)$ is
the manifold given by
\begin{equation}
4 \epsilon^2 \sigma = (1-\epsilon^2-\norm{\xi}^2-\sigma)^2 .
\end{equation}
The map $f$ is a diffeomorphism, and
the two points where $\xi=0$ and
$4 \epsilon^2 \sigma = (1-\epsilon^2-\sigma)^2$ are
elliptic CR singularities
of $f(\partial \Omega)$.
However, all the points where $\sigma=0$ and $\norm{\xi}^2=1-\epsilon^2$
are also CR singularities of $f(\partial \Omega)$.
The extended map $F \colon \overline{\Omega} \to \C^n \times \R$
is not one-to-one, and the boundary of $F(\overline{\Omega})$ is not
the image $f(\partial \Omega)$.  See Figure~\ref{fig:BallSingSq}.

\begin{figure}[ht!]
\includegraphics[scale=1]{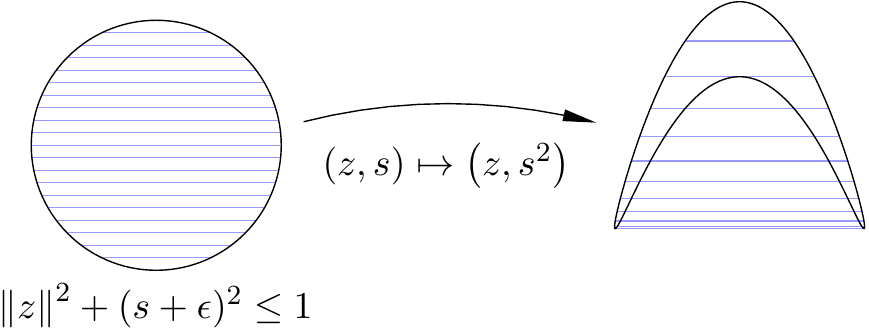}
\caption{Diagram for mapping in Example~\ref{eg:BallSingSq}.\label{fig:BallSingSq}}
\end{figure}
\end{example}

If the image does not necessarily lie in $\C^n \times \R$, then injectivity
may fail inside $\Omega$
even if the image $f(\partial \Omega) = F(\partial \Omega)$ has only
elliptic singularities.

\begin{example}\label{eg:BallSingCubic}
Let 
$\Omega \subset \C^n \times \R$, $n \geq 1$, be the slightly shifted unit
ball
\begin{equation}
\snorm{z}^2 + (s+\epsilon)^2 < 1 ,
\end{equation}
for some small $\epsilon > 0$.
There exists a CR map $F \colon \overline{\Omega} \to \C^{n+1}$ 
such that the
boundary $\partial \Omega$ goes to a CR singular compact manifold diffeomorphically,
$F(\partial \Omega)$ has only elliptic CR singularities, $F$ takes
CR points of $\partial \Omega$ to CR points of $F(\partial \Omega)$,
$F$ extends past the boundary near all points, yet $F$ is not one-to-one on
$\Omega$.

We define $F \colon \overline{\Omega} \to \C^n \times \C$ by
\begin{equation}
(z,s) \mapsto \bigl(z,1-4s^2 +i(8s^3-2s)\bigr) .
\end{equation}
Let $(\xi,\sigma+i\tau) \in \C^n \times \C$ be the coordinates.  The image in the
$\sigma\tau$-plane is a self intersecting curve for $s=\pm \frac{1}{2}$ given by
$\tau^2=\sigma^2(1-\sigma)$. Note that
$F({\partial \Omega})$ does not intersect itself. However,
$F(\Omega)$ does intersect itself when $s = \pm \frac{1}{2}$, so $F(\Omega)$ is
singular.  See Figure~\ref{fig:BallSingCubic}.

\begin{figure}[ht!]
\includegraphics[scale=1]{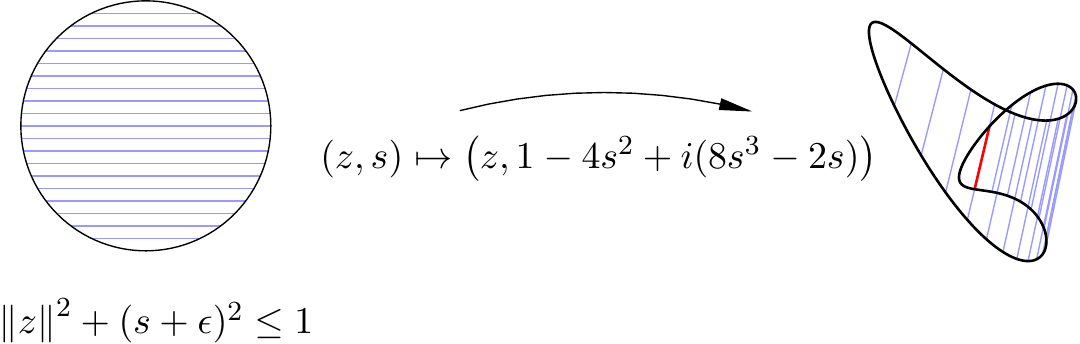}
\caption{Diagram of Example~\ref{eg:BallSingCubic}.\label{fig:BallSingCubic}}
\end{figure}
\end{example}

Finally, we wish to note that not every
$\widetilde{M} \subset \C^{n+1}$ that is
locally a boundary of a Levi-flat hypersurface is, globally, a boundary of
a nonsingular Levi-flat hypersurface with boundary.
In fact, not every singular Levi-flat hypersurface is an image
of a domain in $\C^n \times \R$.

\begin{example}
A Levi-flat singular hypersurface with smooth real-analytic boundary and
isolated CR singular points need not be an image of a domain with boundary
in $\C^n \times \R$.

Let $(\xi,\eta)$ be the coordinates on $\C^2$.  Fix a small $\epsilon > 0$.
Define $H \subset \C^2$ by
\begin{equation} \label{eq:defsingHwithbound}
\Im (\xi^2+\eta^2) = 0, \qquad \sabs{\xi}^2+\sabs{\eta+\epsilon}^2 \leq 1 .
\end{equation}
This hypersurface has an isolated singularity
at the origin.  The boundary $\partial H$ (by which we mean the set where
the inequality in \eqref{eq:defsingHwithbound} is an equality)
is a real submanifold with
isolated CR singularities
(this is the reason for the $\epsilon$), at
\begin{equation}
\Bigl(0,-\epsilon\pm 1\Bigr), \quad
\Bigl(0,\pm i\sqrt{1-\epsilon^2}\Bigr), \quad
\Bigl(\pm i\sqrt{1-\frac{\epsilon^2}{4}},\frac{-\epsilon}{2}\Bigr).
\end{equation}
Through the origin, $H$ contains two ``leaves,'' one given by
$\xi-i\eta=0$, and one given by $\xi+i\eta=0$.

However, $H$ is not an image of a domain in $\C \times \R$:
Suppose it is, that is, suppose 
there exists an open set $U \subset \C \times \R$
and a real-analytic mapping $F \colon U \to H$ that is
holomorphic in the first variable and such that $F(0,0) = (0,0)$,
and such that the image $W \cap H \subset F(U)$ for some neighborhood
$W$ of $0$.
Write $F=(F_1,F_2)$.
For all $(z,s) \in U$,
\begin{equation}
0 = -2i \Im \bigl(F_1(z,s)^2+F_2(z,s)^2\bigr) =
F_1(z,s)^2+F_2(z,s)^2 -
\bar{F}_1(\bar{z},s)^2-\bar{F}_2(\bar{z},s)^2 .
\end{equation}
For each fixed $s$, the map
$z \mapsto F_1(z,s)^2+F_2(z,s)^2$ must be constant,
so each leaf of $\C \times \R$ is mapped to a leaf of $H$.
Suppose without loss of generality that $z \mapsto F(z,0)$ is not
constantly zero and maps to the leaf
$\xi - i\eta = 0$.
Then $z \mapsto F(z,s)$ is not constant 
for all small enough $s$.
For small enough nonzero $s$,
we may also assume that $z \mapsto F(z,s)$
does not map into either one of the two leaves of $H$ through zero.
Consider $z$ is in some small fixed disc $\Delta$ around the origin.
Then $z \mapsto F_1(z,s)+iF_2(z,s)$ has no zero in $\Delta$
for small nonzero $s$ since the 
leaf given by $\xi + i\eta = 0$ only intersects
the leaf given by $\xi-i\eta = 0$.
Via Hurwitz, this means that either
$F_1(z,0)+iF_2(z,0) \equiv 0$, or $F_1(z,0)+iF_2(z,0)$ is never zero
in $\Delta$.  But as the two leaves $\xi + i\eta = 0$ and
$\xi-i\eta = 0$ intersect at the origin, 
$F_1(z,0)+iF_2(z,0)$ has an isolated zero, a contradiction.
\end{example}

\section{Regularity away from CR singularities} \label{section:proofofmain}

In this section we will prove Theorem~\ref{thm:main}.
For simplicity we will sometimes use the following notation.  If $X
\subset \C^n \times \R$, then for any fixed $s \in \R$ we define
\begin{equation}
X_{(s)} = \{ z \in \C^n : (z,s) \in X \} ,
\end{equation}
and we informally refer to $X_{(s)}$ as the \emph{$s$-leaf} of $X$.
Note that as defined $X_{(s)}$ is a subset of $\C^n$.  When it is
necessary to refer to the space $\C^n \times \R$, we
will say that a point $z \in X_{(s)}$ corresponds to the point $(z,s) \in
\C^n \times \R$.

\begin{lemma} \label{lemma:extension}
Let $\Omega \subset \C^n \times \R$, $n \geq 2$, be a bounded domain with
connected real-analytic boundary such that $\partial \Omega$ has only
$A$-nondegenerate 
CR singularities,
and let $f\colon \partial \Omega \to \C^{n+1}$ be a
real-analytic embedding that is CR at CR points of $\partial \Omega$.
Then a real-analytic CR extension $F$ of $f$ to $\overline{\Omega}$ exists;
in fact, there exist an open set $\Upsilon \subset \C^{n+1}$
such that $\overline{\Omega} \subset \Upsilon$ and 
a holomorphic map $G \colon \Upsilon \to \C^{n+1}$ such that
$G|_{\partial \Omega} = f$.
\end{lemma}

\begin{proof}
This is exactly Corollary 1.3 in \cite{crext2}.
\end{proof}

\begin{prop} \label{prop:flatCRsingsonleafs}
Let $\Omega \subset \C^n \times \R$, $n \geq 1$, be a bounded domain with
real-analytic boundary,
and let $\Sigma \subset \partial \Omega$ be the set
of CR singularities.
Let $(z,s) \in \C^n \times \R$ be the coordinates, and let $\pi_{\R} \colon
\C^n \times \R \to \R$ be the projection to the $s$ coordinate.
Then $\pi_{\R}(\Sigma)$ is a finite set.
\end{prop}

\begin{proof}
The set $\Sigma$ is a real-analytic subvariety.  It is precisely the
set where $\partial \Omega$ is tangent to a leaf of $\C^n \times \R$,
that is, a set $\{ (z,s) : s = s_0 \}$.
If $p$ and $q$ are two points on
a connected component of $\Sigma$, then there exists a 
piecewise-smooth path $\gamma \colon [0,1] \to \Sigma$ such that
$\gamma(0) = p$ and $\gamma(1) = q$.
Near a CR singularity, $\partial \Omega$ is a graph $s=\rho(z,\bar{z})$.
Parametrize $\partial \Omega$ locally by $z$ and suppose $\gamma$
is a function valued in the $z$-space. Then $s$ on the path is given by
$\rho \circ \gamma$; but $\nabla \rho = 0$ at CR singularities,
so $s$ is constant on $\gamma$.
Now use the fact that a compact real-analytic subvariety has finitely many topological components.
\end{proof}

\begin{lemma} \label{lemma:FimmersioninOmega}
Let $\Omega \subset \C^n \times \R$, $n \geq 2$, be a bounded domain with
connected real-analytic boundary such that $\partial \Omega$ has only
$A$-nondegenerate 
CR singularities,
and let $f\colon \partial \Omega \to \C^{n+1}$ be a
real-analytic embedding that is CR at CR points of $\partial \Omega$.

If in addition $f$
takes CR points of $\partial \Omega$ to CR points of $f(\partial \Omega)$,
then the extension $F$ from Lemma~\ref{lemma:extension}
is an immersion on ${\overline{\Omega} \setminus \Sigma}$.
Here $\Sigma \subset \partial \Omega$ is the set of CR singularities
of $\partial \Omega$. 
\end{lemma}

\begin{proof}
We use the notation of Lemma 
\ref{lemma:extension}.
By Proposition~\ref{prop:MimmersionatCR}, the map 
$G$ is an immersion at the 
CR points of $\partial\Omega$, and the same is true of $F$.
 Suppose for a contradiction that $F$ is 
not an immersion in $\Omega$.
Let $B \subset \Upsilon$ be the set where $\det DG = 0$,
which is a complex subvariety of dimension $n$.
Then there exists 
$s_0$ such that
$\bigl(\Omega \cap B\bigr)_{(s_0)}$
is nonempty.
Since $F$ is an immersion at CR points of $\partial \Omega$, 
$B \cap \partial \Omega \setminus \Sigma$ is empty.
As $B$ is a complex
variety of dimension $n$, either
$\bigl(\Omega \cap B\bigr)_{(s_0)}$ is of complex dimension $n$ (contains an
open set)
or
$\bigl(\Omega \cap B\bigr)_{(s)}$ is nonempty for all $s$ near $s_0$.
By Proposition~\ref{prop:flatCRsingsonleafs}, $\pi_{\R}(\Sigma)$ is a
finite set.
In particular, for $s$ near $s_0$,
$\bigl(\Omega \cap B\bigr)_{(s)}$ is empty.
Hence
$\bigl(\Omega \cap B\bigr)_{(s_0)}$ contains at least one
of the topological
components of $\Omega_{(s_0)}$.
The dimension of $\partial \bigl( \bigl(\Omega \cap B\bigr)_{(s_0)} \bigr)$ is $2n-1$,
so
$\bigl(\partial \Omega \cap B\bigr)_{(s_0)}$ is 
of dimension at least $2n-1$.  Therefore the dimension of
$\Sigma$ is at least $2n-1$.  However, by Proposition~\ref{prop:sizeofCRsing}
the dimension of $\Sigma$ is at most $n$. This gives a contradiction since  $n \geq 2$ implies $2n-1 > n$.
\end{proof}

Theorem~\ref{thm:main} now follows from the two lemmas.

\section{Regularity at flattenable CR singularities} \label{section:proofofcor}

Let us prove Corollary~\ref{cor:main}.
Suppose that $\Omega$ and $f$ are as in the statement of the corollary.
We will use Theorem~\ref{thm:main} to find an extension $F$, and
then use the result of Fang--Huang~\cite{FangHuang} to flatten the
CR singularities and apply the authors'
previous extension result in \cite{crext2} to find an
extension of the inverse of $f$.  In particular the main point
is the following extension lemma, which follows from combining
Fang--Huang~\cite{FangHuang} and \cite{crext2}.  Recall
that a real-analytic CR submanifold is non-minimal at $q$ if it contains
through $q$ a proper CR submanifold of the same CR dimension.

\begin{thm} \label{thm:fanghuanglnr}
Suppose $M \subset \C^{n+1}$, $n\geq 2$,
is a real-analytic codimension-two submanifold that is
CR singular at $p \in M$, with the CR singularity $A$-nondegenerate
and not the exceptional case (of the form \eqref{eq:exceptionalcase}).
Assume that $M$ is non-minimal at all CR
points.  Let $\varphi \colon M \to \C$ be a real-analytic function
that is CR at the CR points of $M$.  Then there exists a holomorphic
function $\Phi$ defined in a neighborhood $U$ of $p$ in $\C^{n+1}$
such that $\Phi|_{U \cap M} = \varphi|_{U \cap M}$.
\end{thm}

\begin{proof}
Fang--Huang~\cite{FangHuang} prove that such an $M$ can be holomorphically
flattened near $p$, that is, realized as
a subset of $\C^n \times \R$.  Then the authors' result from \cite{crext2},
which is stated in $\C^n \times \R$, obtains the holomorphic extension
$\Phi$.
\end{proof}

\begin{proof}[Proof of Corollary~\ref{cor:main}]
Since $f(\partial \Omega)$ has only $A$-nondegenerate singularities,
Proposition~\ref{prop:imagenotnondegenerate} says that
$f$ must take CR points of $\partial \Omega$ to
CR points of $f(\partial \Omega)$.
On the other hand,
by Corollary~\ref{cor:imageofnondegeneratenotCR}, CR singular points go
to CR singular points.

We apply Theorem~\ref{thm:main} to obtain an extension $F$
that is an immersion on
$\overline{\Omega} \setminus \Sigma$.
What is left to prove is that its
derivative at each CR singularity is nonsingular.

Given a CR point $\widetilde{q} = f(q) \in f(\partial \Omega)$, we 
have that $q$ is a CR point of $\partial \Omega$.
At $q$, $F$ is an immersion, and therefore
the image $F(U)$ of a small neighborhood $U$ of $q$ in $\C^n \times \R$
is a nonsingular Levi-flat hypersurface in $\C^{n+1}$. Hence
$f(\partial \Omega)$ is non-minimal at $\widetilde{q}$. In other words,
$f(\partial \Omega)$ is non-minimal at all CR points.

Fix $p\in \Sigma$.  The point $\widetilde{p} = f(p)$ is a CR singularity of
$f(\partial \Omega)$,
assumed to be $A$-nondegenerate and
not the exceptional case.  Furthermore, $f(\partial \Omega)$ is 
non-minimal at all CR points, so
Theorem~\ref{thm:fanghuanglnr} applies.

Let $\varphi$ be the inverse of $f$ from $f(\partial \Omega)$,
that is, $f \circ \varphi = \operatorname{Id}$.
By Proposition~\ref{prop:CRdiffeo} and Corollary~\ref{cor:imageofnondegeneratenotCR}, $\varphi$ is a CR map on the CR points
of $f(\partial \Omega)$.
By Theorem~\ref{thm:fanghuanglnr},
$\varphi$ extends to a holomorphic map $\Phi$ on a neighborhood
$U \subset \C^n \times \C$
of $\widetilde{p}$.  As it is the unique
extension, $\Phi$ is the inverse of $F$ from
$U \cap F(V\cap\overline{\Omega})$ for some neighborhood $V$ of $p$.  Since $F \circ \Phi = \operatorname{Id}$ on
$U \cap F(V\cap\overline{\Omega})$, the derivative of $\Phi$, and therefore
of $F$, must be nonsingular at $\widetilde{p}$ and $p$, respectively.
\end{proof}

\section{Flat case for CR dimension greater than 1} \label{section:flat}

When the target is $\C^n \times \R$ rather than $\C^{n+1}$, we obtain more
rigidity.  Our  goal here is to prove 
Theorem~\ref{thm:flat} and Corollary~\ref{cor:auto}.

First we show that an extension must map into $\C^n \times \R$, and each
leaf goes to a leaf. This result holds for $n\geq 1$.

\begin{lemma} \label{lemma:mapstoflat}
Suppose $\Omega \subset \C^n \times \R$,
$n \geq 1$, is a bounded domain with real-analytic boundary,
and $F \colon \overline{\Omega} \to \C^{n+1} $ is a real-analytic CR
map such that $F(\partial \Omega)\subset \C^n \times \R$.
Then  $F(\Omega)\subset \C^n \times \R$, and if $(z,s) \in \C^n \times \R$
are the coordinates, then the last component of $F$ depends only on $s$.
\end{lemma}

\begin{proof}
We write $F=(F_1,F_2)$, where $F_2$ is the last component of $F$.
Fix an $s$ such that
$\Omega_{(s)}$ is nonempty.
The set $\Omega_{(s)}$ is bounded,
and $z \mapsto F_2(z,s)$ is real-valued at the boundary
$\partial \Omega_{(s)}$. Hence
the holomorphic function $z \mapsto F_2(z,s)$
must be real-valued on $\Omega_{(s)}$, and as it is holomorphic,
it must be constant.
Thus, $F_2$ is real-valued on $\Omega$ and depends only on $s$.
\end{proof}

To prove that an extension is one-to-one under certain conditions,
we need the following lemma, which is surely classical,
although we did not find a good reference.

\begin{lemma} \label{lemma:onetooneonboundarya}
Suppose $U \subset \C^n$, $n \geq 2$, is a bounded domain,
$\overline{U} \subset W$ for an open set $W$, $\Phi \colon W \to \C^n$ is
holomorphic, $\Phi|_{\partial U}$
is one-to-one, and the Jacobian determinant $\det D\Phi$ is nonzero on
$\partial U$.
Then $\Phi|_{\overline{U}}$ is one-to-one, and in fact
$\Phi$ is a biholomorphic map of a neighborhood of $\overline{U}$.
\end{lemma}

See Lemma~\ref{lemma:onetooneonboundaryn1} for the $n=1$ version.

\begin{proof}
Since $\Phi$ is holomorphic in a neighborhood of $\overline{U}$,
we assume that $\partial U = \partial \overline{U}$.

Let us reduce to the case of connected boundary.  If the boundary is
not connected, then we extend $\Phi$ through the holes using Hartogs'
extension result.  We then apply the lemma using $\Phi$ and the exterior
boundary component of $U$.

The function $\Phi$ has an inverse locally near $\Phi(\partial U)$.
Take a point $q \in \Phi(\partial U)$, and
let $p = \Phi^{-1}(q)$.  Near $p$, $\partial U$ is the limit of
points both in $U$ and outside of $\overline{U}$, and therefore
near $p$, $\partial U$ has 
topological dimension at least $2n-1$
(see e.g.~\cite{HurewiczWallman}*{Theorem IV 4}).
Consequently
$\Phi(\partial U)$ also has dimension at least $2n-1$ near $q$.
Therefore, no nonconstant holomorphic function can vanish on
$\Phi(\partial U)$ in a neighborhood of $q$.
In particular, there is exactly one inverse of $\Phi$ near $q$
that agrees with $\Phi^{-1}$ on $\Phi(\partial U)$.
We therefore have a holomorphic mapping $\Psi$ defined on a neighborhood
of the connected set $\Phi(\partial U)$.
By Hartogs' theorem we find that $\Psi$ extends to the inside of
the bounded hypersurface $\Phi(\partial U)$.

The Jacobian determinant of $\Phi$ never vanishes on $U$ since it does
not vanish on  the boundary and the domain is bounded.
Thus $\Phi$ is an open map, and hence the only 
points of $\overline{U}$ that go to the boundary of
$\Phi(U)$ must be points of $\partial U$.
In particular,
$\Phi$ does not map any point of $U$ to the outside of $\Phi(\partial U)$.
We have that $\Psi \circ \Phi$ is the identity near
$\partial U$ and hence everywhere in $U$, so $\Phi$ is a biholomorphism (in
a possibly smaller neighborhood $W$).
\end{proof}

Another result for which we could not find a good reference
(except as an exercise) is the following version of a well-known
one-variable result.

\begin{prop} \label{prop:limitonetoone}
Suppose $U \subset \C^n$ is open and $\Phi_k \colon U \to \C^n$ a sequence
of holomorphic mappings that are one-to-one and converge uniformly on
compact sets to $\Phi \colon U \to \C^n$.  Then either
$\Phi$ is one-to-one or
$\det D\Phi$ vanishes on some topological component of $U$.
\end{prop}

\begin{proof}
As $\det D\Phi$ is a single holomorphic function we apply Hurwitz to find
that either it vanishes identically on some component of $U$ or it 
vanishes nowhere.  Suppose it vanishes nowhere.
Suppose for a contradiction that
$\Phi(p_1) = q$, 
$\Phi(p_2) = q$, and $p_1 \not= p_2$.  Each $\Phi_k$ is a diffeomorphism
and $\Phi$ is a local diffeomorphism and an open map.
Take a small enough open ball $B_\epsilon(q)$ (of radius $\epsilon > 0$
around $q$)
such that $\Phi^{-1}\bigl(B_\epsilon(q)\bigr)$ contains
two disjoint diffeomorphic images of $B_\epsilon(q)$, call them
$\sB_1$ and $\sB_2$, with
$p_1 \in \sB_1$ and
$p_2 \in \sB_2$.
For $k$ large enough $\Phi_k$ takes the boundary of $\sB_1$ close
enough to the boundary of $B_\epsilon(q)$ such that
$B_{\epsilon/2}(q) \subset \Phi_k(\sB_1)$.  
Then $\Phi_k(p_2)$ is more than $\epsilon/2$ away from $q$,
which is a contradiction.
\end{proof}

We can now prove that an extension is one-to-one 
if it maps CR points of the boundary to CR points.

\begin{lemma} \label{lemma:onetoonea}
Suppose $\Omega \subset \C^n \times \R$, $n \geq 2$,
is a bounded domain with connected
real-analytic boundary
such that $\partial \Omega$ has only $A$-nondegenerate CR singularities.
Suppose $F \colon \overline{\Omega} \to \C^n \times \R$ is a real-analytic CR
map such that $F|_{\partial \Omega}$ is a real-analytic
embedding that takes CR points of $\partial \Omega$ to
CR points of $F(\partial \Omega)$.
Then $F$ is one-to-one on $ \overline{\Omega}$.
\end{lemma}

\begin{proof}
Proposition~\ref{prop:MimmersionatCR} applies since  
$F$ extends holomorphically to a
neighborhood of $\overline{\Omega}$ in $\C^{n+1}$. Thus,
$F$ is an immersion at CR points of the boundary.

As in the proof of Lemma~\ref{lemma:mapstoflat}, we write 
$F = (F_1,F_2)$ and find that $F_2$ is a real-valued function
that depends on $s$ only. Since $F$ is an immersion at CR points of the boundary and
 $F_2$ depends on $s$ only, $F_1$ restricted to a leaf
has an invertible derivative at CR points of the boundary.

As in Proposition~\ref{prop:flatCRsingsonleafs}, 
let $\pi_{\R}(\Sigma) \subset \R$ be the set of $s$ that correspond to the CR
singularities.  Then for any $s \notin \pi_{\R}(\Sigma)$
such that $\Omega_{(s)}$ is nonempty,
we find ourselves in the situation of
Lemma~\ref{lemma:onetooneonboundarya}.
Therefore, $F_1$, and hence $F$,
is one-to-one when
restricted to a leaf $s \notin \pi_{\R}(\Sigma)$.
That is, if $F(z,s) = F(z',s)$ then $z=z'$.

Now take an exceptional leaf, that is,
assume that $s_0 \in \pi_{\R}(\Sigma)$.
By Proposition~\ref{prop:flatCRsingsonleafs} there exists
a sequence $\{ s_k \}$ where $s_k \to s_0$ and
the maps $\Phi_k(z) = F_1(z,s_k)$ are one-to-one.
Take the largest open set
$U \subset \Omega_{(s_0)}$ such that $\Phi_k(z)$ is defined
for all $k$, that is, $U = \bigcap_k \Omega_{(s_k)}$.
Apply Proposition~\ref{prop:limitonetoone} to find that
either $\Phi(z) = F_1(z,s_0)$ is one-to-one on $U$
or $\det D \Phi$ vanishes on some component of $U$.
  Taking all possible tails of
the sequence, such $U$ fill
$\Omega_{(s_0)}$.  So
$\Phi(z)$ is one-to-one on $\Omega_{(s_0)}$
or $\det D\Phi$ vanishes on some component.
If $\det D\Phi$ vanishes on some component, then
every point of $\partial \Omega_{(s_0)}$ must correspond
to CR singular points of $\partial \Omega$.  Then the dimension
of the CR singular points is at least $2n-1$, but
Proposition~\ref{prop:sizeofCRsing} says that
the dimension of $\Sigma$ is at most $n$.
As $n \geq 2$, then $2n-1 > n$ obtains a contradiction.
Therefore, $\Phi$ is one-to-one, or in other words,
if $F(z,s_0) = F(z',s_0)$ then $z=z'$.

Assume for a contradiction that there exist $(z,s_1)$ and $(z',s_2)$
such that $s_1\neq s_2$ but $F(z,s_1) = F(z',s_2)$.
By the mean value theorem applied to $F_2$ there exists 
$s_0$ between $s_1$ and $s_2$ such that
$F_2'(s_0) = 0$.  As $F$ restricted to $\partial \Omega$ is a
diffeomorphism, this means that the image of $\partial \Omega$ near every
$(z'',s_0) \in \partial \Omega$ must be tangent to a leaf, that is,
$F(z'',s_0)$ is a CR singular point of $F(\partial \Omega)$.
This would imply that the CR singular points of $F(\partial \Omega)$
have dimension at least $2n-1$.
As CR points go to CR points, the CR singular points of $F(\partial \Omega)$
must lie in $F(\Sigma)$ (in fact the two sets are equal).
By
Proposition~\ref{prop:sizeofCRsing} the dimension of $\Sigma$ (and hence of
$F(\Sigma)$) is at most $n$, and
therefore the dimension of the set of CR singular points of
$F(\partial \Omega)$ is also at most $n$.
As $n \geq 2$, then $2n-1 > n$ obtains a contradiction.
\end{proof}

We now show that the extension is an immersion at the CR 
singular points if the image has only $A$-nondegenerate
CR singularities and $n\geq 2$. 

\begin{lemma} \label{lemma:embedatsing}
Suppose $\Omega \subset \C^n \times \R$, $n \geq 2$,
is a bounded domain with connected
real-analytic boundary
such that $\partial \Omega$ has only $A$-nondegenerate CR singularities.
Suppose $F \colon \overline{\Omega} \to \C^n \times \R$ is a 
real-analytic CR
map such that $F|_{\partial \Omega}$ is a real-analytic
embedding whose image has only $A$-nondegenerate
CR singularities. Then $F$ is an immersion at each CR 
singular point of $\partial{\Omega}$.
\end{lemma}

\begin{proof}
The proof follows in the same way as Corollary~\ref{cor:main},
but here we already have a flat image, and $F$ has an inverse.

By Proposition~\ref{prop:imagenotnondegenerate},
$F$ takes CR points to CR points, so we use Lemma~\ref{lemma:onetoonea}. 
Therefore, $F$ is one-to-one on $\overline{\Omega}$. 
The restriction of $F^{-1}$ to $F(\partial \Omega)$ is a 
real-analytic map that is CR at
CR points since $F|_{\partial \Omega}$ is an embedding and 
both Proposition~\ref{prop:CRdiffeo} and
Corollary~\ref{cor:imageofnondegeneratenotCR} apply.
Therefore,
as before (using results of \cite{crext2})
$F^{-1}|_{F(\partial \Omega)}$ extends to a real-analytic
CR map on $F(\overline{\Omega})$.
As the extension
is unique, it must equal $F^{-1}$ on $F(\overline{\Omega})$.  

Now let $p\in\partial{\Omega}$ be a CR singular point. Then $F^{-1}$ is differentiable at $F(p)$,  so
the derivative of $F$ at $p$ is nonsingular.
\end{proof}

Theorem~\ref{thm:flat} follows from the preceding results: Let $n\geq 2$, 
and let $\Omega \subset \C^n \times \R$ 
be a bounded domain with connected real-analytic boundary
such that $\partial \Omega$ has only $A$-nondegenerate CR singularities.
Let $f\colon \partial \Omega \to \C^n \times \R$ be a real-analytic
embedding that is CR at CR points of $\partial \Omega$.
It follows from Lemma~\ref{lemma:extension} that there exists
a real-analytic CR map $F \colon \overline{\Omega} \to \C^{n+1}$
such that $F|_{\partial \Omega} = f$. By Lemma~\ref{lemma:mapstoflat}, $F$ maps into
$\C^n \times \R$. 

To prove part (a) of the theorem, assume that $f$ takes CR points
of $\partial \Omega$ to CR points of $f(\partial \Omega)$. By Theorem~\ref{thm:main},
$F|_{\overline{\Omega}\setminus \Sigma}$ is an immersion. (Here
 $\Sigma \subset \partial \Omega$ is the set of CR singularities of $\partial \Omega$.)
 By  Lemma~\ref{lemma:onetoonea},
$F$ is one-to-one on $\overline\Omega$.

To prove part (b), assume that $f(\partial\Omega)$ has only  
$A$-nondegenerate CR singularities.
By Proposition~\ref{prop:imagenotnondegenerate},
$f$ takes CR points to CR points, so we apply (a) to get that $F$ is one-to-one
and that $F|_{\overline{\Omega}\setminus \Sigma}$ is an immersion. 
By Lemma~\ref{lemma:embedatsing}, $F$ is an immersion on $\Sigma$ as well.

\begin{proof}[Proof of Corollary~\ref{cor:auto}]
Since $f(\partial \Omega) = \partial \Omega$, the image has only
$A$-nondegenerate singularities.  By part (b) of Theorem~\ref{thm:flat},  the extension $F$ is one-to-one and is an embedding on $\overline{\Omega}$. Now use Proposition~\ref{prop:equivcr}.
\end{proof}

\section{Flat case in CR dimension 1} \label{section:flatn1}

Let us now prove Theorem~\ref{thm:flatn1}. In outline the argument is similar to that in the preceding section. First, though, we prove the existence of an extension and list some of the topological consequences of the ellipticity assumption. 

\begin{lemma} \label{lemma:extension1}
Let $\Omega \subset \C \times \R$ be a bounded domain with real-analytic boundary
such that $\partial \Omega$ has only elliptic  CR singularities.
Let $f\colon \partial \Omega \to \C \times \R$ be a real-analytic
map.
Suppose for every $c \in \R$ such that $\Omega \cap \{ s = c \}$ is nonempty,
there exists a continuous map on $\overline{\Omega} \cap \{ s = c \}$,
holomorphic on $\Omega \cap \{ s = c \}$,
extending $f|_{\partial \Omega \cap \{ s = c \}}$.

Then there exists a real-analytic CR map
$F \colon \overline{\Omega} \to \C \times \R$
such that $F|_{\partial \Omega} = f$. 
\end{lemma}

\begin{proof}
Note that the condition on $f$ is precisely what is required in the
main result of \cite{crext1} to obtain an extension $F \colon
\overline{\Omega} \to \C^{2}$ that is real-analytic
and CR.   By Lemma~\ref{lemma:mapstoflat}, $F$ maps into
$\C \times \R$. 
\end{proof}

One immediate consequence of the following result is that the boundary is connected.

\begin{lemma}\label{lemma:topologyOmega}
Suppose $\Omega \subset \C^n \times \R$, $n \geq 1$, is
a bounded domain with smooth boundary such that
$\partial \Omega$ has only elliptic CR singularities.

Then
$\partial \Omega$ is homeomorphic to a sphere, 
$\partial \Omega$ has exactly two CR singularities,
corresponding to the minimum and the maximum value of $s$
on $\partial \Omega$,
and for each $s \in \R$, $\Omega_{(s)}$ is either empty
or a bounded domain with
smooth connected boundary.
In particular, if $\Omega_{(s)}$ is nonempty, then
no point in $\partial \Omega_{(s)}$ corresponds to a CR singularity
of $\partial \Omega$.
\end{lemma}

\begin{proof}
See Proposition~3.2 in \cite{crext1}.
\end{proof}

The following classical result follows from the argument principle.

\begin{lemma} \label{lemma:onetooneonboundaryn1}
Suppose $U \subset \C$ is a bounded domain with smooth connected boundary,
$\overline{U} \subset W$ for an open set $W$, $\Phi \colon W \to \C$ is
holomorphic, and $\Phi|_{\partial U}$
is one-to-one.
Then $\Phi|_{\overline{U}}$ is one-to-one.
\end{lemma}

The hypothesis of connected
boundary is necessary:
Consider $z \mapsto z + \frac{1}{z}$ defined on an
annulus $U$ centered at zero with inner radius $r < 1$ and
outer radius $R > 1$, where $\frac{1}{R} \not= r$.

\begin{lemma} \label{lemma:onetoonen1}
Suppose $\Omega \subset \C \times \R$
is a bounded domain with 
real-analytic boundary
such that $\partial \Omega$ has only elliptic CR singularities.
Suppose $F \colon \overline{\Omega} \to \C \times \R$ is a real-analytic CR
map such that $F|_{\partial \Omega}$ is a real-analytic
embedding that takes CR points of $\partial \Omega$ to
CR points of $F(\partial \Omega)$.
Then $F$ is one-to-one on $ \overline{\Omega}$.
\end{lemma}

\begin{proof}
The proof is similar to that of Lemma~\ref{lemma:onetoonea}.
We again write 
$F = (F_1,F_2)$ and find via
Lemma~\ref{lemma:mapstoflat} that the real-valued function $F_2$ depends on $s$ only. 

Lemma~\ref{lemma:topologyOmega} says that
if $\Omega_{(s)}$ is nonempty then
$\Omega_{(s)}$ has a smooth connected boundary
and $\partial \Omega_{(s)}$ has no points that correspond to CR
singularities
of $\partial \Omega$.
In particular, if $\Omega_{(s)}$ is nonempty,
we apply
Lemma~\ref{lemma:onetooneonboundaryn1}
and conclude that 
$F_1$, and hence $F$,
is one-to-one when
restricted to a leaf.
That is, if $F(z,s) = F(z',s)$ then $z=z'$.

Recall that $\partial \Omega$ has exactly two CR
singularities and these correspond to
the extremal values of $s$. Assume for a contradiction that there exist $(z,s_1)$ and $(z',s_2)$
such that $s_1\neq s_2$ but $F(z,s_1) = F(z',s_2)$.
By the mean value theorem applied to $F_2$ there exists 
$s_0$ between $s_1$ and $s_2$ such that
$F_2'(s_0) = 0$.  As $F$ restricted to $\partial \Omega$ is a
diffeomorphism, this means that the image of $\partial \Omega$ near every
$(z'',s_0) \in \partial \Omega$ must be tangent to a leaf, that is,
$F(z'',s_0)$ is a CR singular point of $F(\partial \Omega)$.
But as CR points go to CR points, the only CR singular points of
$F(\partial \Omega)$ can be the images of the two poles, which is
a contradiction.
\end{proof}

In Theorem~\ref{thm:flatn1} we assume that $f$ is a
real-analytic embedding that takes CR points of the boundary to CR points. 
By Lemma~\ref{lemma:onetoonen1}, the extension $F$ obtained from Lemma \ref{lemma:extension1} is one-to-one.
Also, Proposition~\ref{prop:MimmersionatCR} 
implies that $F$ is an immersion at CR points of the boundary. 
The fact that $F$ is an immersion inside holds in more generality, so we
state it separately.  It  requires only that $\Omega$ be a bounded domain with real-analytic boundary
and that $\partial \Omega$ and $F(\partial \Omega)$ have
isolated CR singularities. (Proposition~\ref{prop:sizeofCRsing} 
implies that elliptic singularities
are always isolated.)

\begin{lemma}
Let $\Omega \subset \C \times \R$ be a bounded domain with real-analytic boundary
such that $\partial \Omega$ has isolated (hence finitely many) CR singularities.
Suppose that $F \colon \overline{\Omega} \to \C \times \R$
is a real-analytic CR map such that $F|_{\partial \Omega}$ is a
real-analytic embedding and $F(\partial \Omega)$
has only isolated CR singularities.
Then $F|_{\Omega}$ is an immersion.
\end{lemma}

\begin{proof}
As before we write $F = (F_1,F_2)$   and find that the real-valued function $F_2$ depends on $s$ only.  
By complexifying we consider $F$ to be a map from a neighborhood of
$\overline{\Omega}$ in $\C^2$ to $\C^2$ that
happens to map $\overline{\Omega}$ to $\C \times \R$.
The complex Jacobian determinant of $F$ is
\begin{equation}
\det DF = \frac{\partial F_1}{\partial z} \frac{\partial F_2}{\partial s} .
\end{equation}

Let $I = \pi_{\R}(\Omega)$ (projection onto the $s$ coordinate)
be the interval consisting of those $s$ that correspond to points
in $\Omega$.
For each $s \in I$, the open set $\Omega_{(s)}$ is nonempty, and hence
its boundary has infinitely many points.  Hence for every $s \in I$
there exists at least one point $(z,s) \in \partial \Omega$ that is
a CR point of $\partial \Omega$ and such that $F(z,s)$ is a CR point
of $F(\partial \Omega)$.
By Proposition~\ref{prop:MimmersionatCR},
$F$ is an immersion at such a point $(z,s)$.
  Therefore,
$\frac{\partial F_2}{\partial s}\neq 0$ for all $s \in I$
(using the fact that $F_2$ does not depend on $z$).
Suppose $F$ is not an immersion at some point $(z_0,s_0)  \in \Omega$.
Then $\frac{\partial F_1}{\partial z}(z_0,s_0) = 0$.  By the Hurwitz theorem
$\frac{\partial F_1}{\partial z}$ must have a zero in $\Omega_{(s)}$ for all
$s \in I$.  Thus there must exist such a point $(z_1,s_1)$ on the boundary
such that $s_1 \in \partial I$.
Then $(z_1,s_1)$ is a CR singular point, and as
$\frac{\partial F_1}{\partial z}(z_1,s_1) = 0$,
this contradicts $f$ being a diffeomorphism.
\end{proof}

Note that the hypotheses are necessary:  For the assumption of boundedness, consider the
function $F(z,s) = \bigl(\sinh(z),s\bigr)$, and let $\Omega = \{ (z,s) : \Im z >
0 \}$.  Then $F$ restricted to $\partial \Omega$ is a diffeomorphism onto
its image, but $\det DF$ is zero whenever $\cosh(z) = 0$.

For the  assumption of isolated CR singularities,
consider the function $F(z,s) = \bigl(z,{(s-1)}^3\bigr)$,
and let $\Omega = \{ (z,s) : s^2+(\sabs{z}^2-1)^3 < 1 \}$.
Then $\partial \Omega$ is a real-analytic submanifold with
elliptic CR singularities at $(0,\pm \sqrt{2})$.
But $\partial \Omega$ has further CR singularities at
all points where $\sabs{z}=1$ and $s = \pm 1$.  The function $F$ restricted
to $\partial \Omega$ is an embedding, and $F$ is one-to-one on $\Omega$, but
$F$ is not an immersion at all points where $s=1$.

Now we prove a version of Lemma~\ref{lemma:embedatsing} that holds for $n=1$.

\begin{lemma} \label{lemma:embedatsingn1}

Suppose that $\Omega \subset \C \times \R$ is a bounded domain with real-analytic boundary
such that $\partial \Omega$ has only elliptic CR singularities. Let 
$F \colon \overline{\Omega} \to \C \times \R$ be a 
real-analytic CR
map such that $F|_{\partial \Omega}$ is a real-analytic
embedding. Assume that $F$ takes CR points of $\partial \Omega$ to
CR points of $F(\partial \Omega)$ and that $F(\partial \Omega)$ has only elliptic 
CR singularities. Then $F$ is an immersion at each CR
singular point of $\partial{\Omega}$.
\end{lemma}

\begin{proof}
 By Lemma~\ref{lemma:onetoonen1}, $F$ is one-to-one on $ \overline{\Omega}$. 
The map $F^{-1}$ is continuous on $F(\overline \Omega)$, so $F^{-1}|_{F(\partial \Omega)}$ extends
continuously along each leaf in $F(\Omega)$ to a holomorphic map.
Because $F(\partial\Omega)$ has only elliptic singularities, by the main result of
\cite{crext1},
$F^{-1}|_{F(\partial \Omega)}$ extends to a real-analytic
CR map on $F(\overline{\Omega})$.
As in the proof of Lemma~\ref{lemma:embedatsing}, it follows that
the derivative of $F$ is nonsingular at each CR singular point.
\end{proof}

Theorem~\ref{thm:flatn1} follows easily from the preceding results.

\medskip

Finally let us prove Corollary~\ref{cor:flatn1}; in particular,
let us prove the following lemma.

\begin{lemma}
Let
$\Omega \subset \C \times \R$ be a bounded domain with
real-analytic boundary
such that $\partial \Omega$ has only elliptic 
CR singularities.
Let $F\colon \overline{\Omega} \to \overline{\Omega}$
be a real-analytic map such that $F|_{\partial \Omega}$ is a diffeomorphism
onto $\partial \Omega$.

Then $F$ takes the CR singularities of $\partial \Omega$ to themselves,
possibly interchanging them (and hence $F$ takes
CR points to CR points).
\end{lemma}

\begin{proof}
By Lemma~\ref{lemma:topologyOmega}
we find that the $s$ that correspond to CR singularities of $\partial
\Omega$
are the extrema.  Without loss of generality (moving $s$ by an affine transform)
suppose that $\Omega_{(s)}$
is nonempty for $s \in (-1,1)$ and $s = \pm 1$ corresponds to the two
CR singularities of $\partial \Omega$.  Also, without loss of generality
(moving $z$ by an affine transform) we assume that
the CR singularities are at $(0,\pm 1)$.
Because (as above) $F_2$ depends on $s$ only,  at the CR singularities
the derivative of $F$ in the $(z,s)$ variables (really the derivative of the
complexification of $F$) is
\begin{equation}
\begin{bmatrix}
\frac{\partial F_1}{\partial z} &
\frac{\partial F_1}{\partial s} \\
0 &
\frac{\partial F_2}{\partial s}
\end{bmatrix} .
\end{equation}
The tangent space of $\partial \Omega$ at the CR singularities $(0,\pm 1)$
is orthogonal to
$\frac{\partial}{\partial s}$.
Therefore, since
$F|_{\partial \Omega}$ is a diffeomorphism,
the pushforward of this tangent space is the tangent space of $F(\partial
\Omega)$ at $F(0,\pm 1)$. Hence this space must also
be orthogonal to
$\frac{\partial}{\partial s}$.  In other words,
$F(0,\pm 1)$ is also one of
the CR singular points $(0,\pm 1)$.
\end{proof}

 We note that the proof of the lemma also works for
$n \geq 2$, although when $n \geq 2$ we also have
Proposition~\ref{prop:imagenotnondegenerate}
and Corollary \ref{cor:imageofnondegeneratenotCR}.
Thus, in the
higher-dimensional
case the result is local, while a global argument must be made for
$n=1$.

The corollary now follows using arguments similar to those in the proof of Corollary~\ref{cor:auto}.


\def\MR#1{\relax\ifhmode\unskip\spacefactor3000 \space\fi%
  \href{http://www.ams.org/mathscinet-getitem?mr=#1}{MR#1}}



\begin{bibdiv}
\begin{biblist}


\bib{bedford:boundaries}{article}{
   author={Bedford, Eric},
   title={Levi flat hypersurfaces in ${\bf C}\sp{2}$ with prescribed
   boundary: stability},
   journal={Ann.\ Scuola Norm.\ Sup.\ Pisa Cl.\ Sci.\ (4)},
   volume={9},
   date={1982},
   number={4},
   pages={529--570},
   issn={0391-173X},
   review={\MR{0693779}},
}

\bib{BG:envhol}{article}{
      author={Bedford, Eric},
      author={Gaveau, Bernard},
       title={Envelopes of holomorphy of certain {$2$}-spheres in {${\bf
  C}\sp{2}$}},
        date={1983},
        ISSN={0002-9327},
     journal={Amer.\ J.\ Math.},
      volume={105},
      number={4},
       pages={975\ndash 1009},
      review={\MR{0708370}},
}


\bib{Bishop:diffman}{article}{
      author={Bishop, Errett},
       title={Differentiable manifolds in complex {E}uclidean space},
        date={1965},
        ISSN={0012-7094},
     journal={Duke Math.\ J.},
      volume={32},
       pages={1\ndash 21},
      review={\MR{0200476}},
}



\bib{DTZ}{article}{
   author={Dolbeault, Pierre},
   author={Tomassini, Giuseppe},
   author={Zaitsev, Dmitri},
   title={On boundaries of Levi-flat hypersurfaces in ${\mathbb C}^n$},
   language={English, with English and French summaries},
   journal={C.\ R.\ Math.\ Acad.\ Sci.\ Paris},
   volume={341},
   date={2005},
   number={6},
   pages={343--348},
   issn={1631-073X},
   review={\MR{2169149}},
}

\bib{DTZ2}{article}{
   author={Dolbeault, Pierre},
   author={Tomassini, Giuseppe},
   author={Zaitsev, Dmitri},
   title={Boundary problem for Levi flat graphs},
   journal={Indiana Univ.\ Math.\ J.},
   volume={60},
   date={2011},
   number={1},
   pages={161--170},
   issn={0022-2518},
   review={\MR{2952414}},
}

\bib{FangHuang}{article}{
   author={Fang, Hanlong},
   author={Huang, Xiaojun},
   title={Flattening a non-degenerate CR singular point of real codimension
   two},
   journal={Geom.\ Funct.\ Anal.},
   volume={28},
   date={2018},
   number={2},
   pages={289--333},
   issn={1016-443X},
   review={\MR{3788205}},
}

\bib{HarveyLawson}{article}{
   author={Harvey, F.\ Reese},
   author={Lawson, H.\ Blaine, Jr.},
   title={On boundaries of complex analytic varieties.\ I},
   journal={Ann.\ of Math.\ (2)},
   volume={102},
   date={1975},
   number={2},
   pages={223--290},
   issn={0003-486X},
   review={\MR{0425173}},
}

\bib{huangkrantz}{article}{
   author={Huang, Xiao Jun},
   author={Krantz, Steven G.},
   title={On a problem of Moser},
   journal={Duke Math.\ J.},
   volume={78},
   date={1995},
   number={1},
   pages={213--228},
   issn={0012-7094},
   review={\MR{1328757}},
}


\bib{HuangYin:flatI}{article}{
   author={Huang, Xiaojun},
   author={Yin, Wanke},
   title={Flattening of CR singular points and analyticity of the local hull
   of holomorphy I},
   journal={Math.\ Ann.},
   volume={365},
   date={2016},
   number={1-2},
   pages={381--399},
   issn={0025-5831},
   review={\MR{3498915}},
}

\bib{HurewiczWallman}{book}{
   author={Hurewicz, Witold},
   author={Wallman, Henry},
   title={Dimension Theory},
   series={Princeton Mathematical Series, v.\ 4},
   publisher={Princeton University Press, Princeton, N.\ J.},
   date={1941},
   pages={vii+165},
   review={\MR{0006493}},
}

\bib{crext1}{article}{
   author={Lebl, Ji\v r\'\i },
   author={Noell, Alan},
   author={Ravisankar, Sivaguru},
   title={Extension of CR functions from boundaries in ${\mathbb C}^n\times{\mathbb R}$},
   journal={Indiana Univ.\ Math.\ J.},
   volume={66},
   date={2017},
   number={3},
   pages={901--925},
   issn={0022-2518},
   review={\MR{3663330}},
}

\bib{crext2}{article}{
   author={Lebl, Ji\v r\'\i },
   author={Noell, Alan},
   author={Ravisankar, Sivaguru},
   title={Codimension two CR singular submanifolds and extensions of CR
   functions},
   journal={J.\ Geom.\ Anal.},
   volume={27},
   date={2017},
   number={3},
   pages={2453--2471},
   issn={1050-6926},
   review={\MR{3667437}},
}

\bib{crext3}{article}{
   author={Lebl, Ji\v r\'\i },
   author={Noell, Alan},
   author={Ravisankar, Sivaguru},
   title={On Lewy extension for smooth hypersurfaces in $\C^n\times\R$},
   journal={Trans.\ Amer.\ Math.\ Soc., to appear},
   note={\href{https://arxiv.org/abs/1704.08662}{arXiv:1704.08662}},
   date={},
}

\bib{MW:normal}{article}{
      author={Moser, J{\"u}rgen~K.},
      author={Webster, Sidney~M.},
       title={Normal forms for real surfaces in {${\bf C}\sp{2}$} near complex
  tangents and hyperbolic surface transformations},
        date={1983},
        ISSN={0001-5962},
     journal={Acta Math.},
      volume={150},
      number={3--4},
       pages={255--296},
      review={\MR{0709143}},
}
		


\end{biblist}
\end{bibdiv}
\end{document}